\documentclass{amsart}
\usepackage{amsmath, amssymb, bbm}
\usepackage{color}
\usepackage{graphicx, epstopdf}

\usepackage{todonotes}

\theoremstyle{plain}
\newtheorem{theorem}{Theorem}[section]
\newtheorem{proposition}[theorem]{Proposition}
\newtheorem{corollary}[theorem]{Corollary}
\newtheorem{lemma}[theorem]{Lemma}

\theoremstyle{definition}
\newtheorem{definition}[theorem]{Definition}

\theoremstyle{remark}
\newtheorem*{remark}{Remark}
\theoremstyle{remark}

\numberwithin{equation}{section}

\DeclareMathOperator{\LCD}{LCD}

\DeclareMathOperator{\Spread}{Spread}

\DeclareMathOperator{\dist}{dist}

\DeclareMathOperator{\loc}{Loc}

\def \N {\mathbb{N}}
\def \R {\mathbb{R}}
\def \C {\mathbb{C}}

\def \Z {\mathbb{Z}}
\def \E {\mathbb{E}}
\def \P {\mathbb{P}}

\def \DD {\mathcal{D}}
\def \EE {\mathcal{E}}
\def \II {\mathfrak{Im}}
\def \NN {\mathcal{N}}

\def \LL {\mathcal{L}}

\def \RR {\mathfrak{Re}}

\def \a {\alpha}
\def \b {\beta}
\def \g {\gamma}
\def \e {\varepsilon}
\def \eps {\varepsilon}
\def \d {\delta}

\def \l {\lambda}

\def \< {\langle}
\def \> {\rangle}
\def \^ {\widehat}

\def \dist {{\rm dist}}

\def \Span {{\rm span}}

\def \Prob {{\mathbb{P}}}

\def \supp {{\rm supp}}

\def \Comp {{\mathit{Comp}}}
\def \Incomp {{\mathit{Incomp}}}

\newcommand{\norm}[1]{\left \| #1 \right \|}
\def \etc {,\ldots,}

\begin{document}

\title[Eigenvector delocalization for non-Hermitian random matrices]{Eigenvector delocalization for non-Hermitian random matrices and applications}

\author{Kyle Luh \and Sean O'Rourke}

\address{
  Center of Mathematical Sciences and Applications, Harvard University}
\email{kluh@cmsa.fas.harvard.edu}

\address{Department of Mathematics, University of Colorado at Boulder, Boulder, CO 80309   }
\email{sean.d.orourke@colorado.edu}

\thanks{K. Luh has been supported in part by the National Science Foundation under Award
No. 1702533}

\thanks{S. O'Rourke has been supported in
part by NSF grants ECCS-1610003 and DMS-1810500.}

\date{\today}

\begin{abstract}
Improving upon results of Rudelson and Vershynin, we establish delocalization bounds for eigenvectors of independent-entry random matrices.  In particular, we show that with high probability every eigenvector is delocalized, meaning any subset of its coordinates carries an appropriate proportion of its mass.  Our results hold for random matrices with genuinely complex as well as real entries.  As an application of our methods, we also establish delocalization bounds for normal vectors to random hyperplanes.  The proofs of our main results rely on a least singular value bound for genuinely complex rectangular random matrices, which generalizes a previous bound due to the first author, and may be of independent interest.  
\end{abstract}

\maketitle

\section{Introduction}

Let $G$ be an $n \times n$ random matrix with independent and identically distributed (iid) entries whose real and imaginary parts are independent standard normal random variables.  It is not difficult to see that the distribution of $G$ is invariant under multiplication (either on the right or left) by unitary matrices.  Among others, this implies that the unit eigenvectors of $G$ are uniformly distributed on the complex unit sphere $S_{\mathbb{C}}^{n-1}$.  

For an $n \times n$ independent-entry matrix $A$ with non-Gaussian entries no such invariance property exists, and the distribution of the eigenvectors is not easily described.  In fact, if the entries of $A$ are discrete random variables, then the eigenvectors cannot have continuous distribution.  However, the universality phenomenon in random matrix theory asserts that, under some appropriate regularity conditions on the entries, the eigenvectors of $A$ should be approximately uniform on the unit sphere for large enough dimension $n$.  As such, we expect each eigenvector of $A$ to be have asymptotically the same properties as a vector uniformly distributed on the unit sphere.  

The goal of this note is to quantify some of these properties for the eigenvectors of $A$.  
Let us begin by recalling some delocalization properties for random vectors uniformly distributed on the unit sphere.  To fix some notation, for a vector $v = (v_i)_{i=1}^n \in \mathbb{C}^n$, we let $\|v\|_{\infty}$ denote the $\ell^\infty$-norm of $v$ and $\|v\|_2$ denote the $\ell^2$-norm of $v$.  In addition, for $I \subset [n] := \{1, \ldots, n\}$, we let $v_I$ denote the $|I|$-vector $v_I = (v_i)_{i \in I}$.  Here, $|I|$ denotes the cardinality of the finite set $I$.  

\begin{proposition}[Largest coordinate of a uniformly distributed vector on the unit sphere] \label{prop:largest-coord}
Let $v$ be uniformly distributed on the unit sphere in $\mathbb{C}^n$ or $\mathbb{R}^n$.  Then there exists an absolute constant $C > 0$ such that
\begin{equation} \label{eq:ellinfinity}
	\|v\|_{\infty} \leq C \sqrt{ \frac{\log n}{n}} 
\end{equation}
with probability $1-o(1)$\footnote{Here $o(1)$ denotes a term which tends to zero as $n \to \infty$.  See Section \ref{sec:notation} for a complete description of our asymptotic notation.}.  
\end{proposition}

The bound on the $\ell^\infty$-norm in \eqref{eq:ellinfinity} rules out peaks in the distribution of mass of $v$.  This bound is optimal, up to the choice of constant $C$.  A similar bound was recently extended to eigenvectors of matrices with independent subgaussian entries \cite{RVdeloc}.  

\begin{definition}[Subgaussian random variable]
A real random variable $X$ is called \emph{subgaussian} if there exists $B > 0$ called the subgaussian moment of $X$ such that
\[ \Prob( |X| > t) \leq 2 e^{-t^2 / B^2} \]
for every $t > 0$.  
\end{definition}

\begin{theorem}[Theorem 1.1, \cite{RVdeloc}] \label{thm:RVdeloc}
Let $A$ be an $n \times n$ matrix whose entries $a_{ij}$ are independent real-valued random variables with mean zero, unit variance, and subgaussian moment bounded by $B$.  Let $t \geq 2$.  Then with probability at least $1 - n^{1 - t}$, every eigenvector $v$ of $A$ satisfies
\[ \|v \|_{\infty} \leq \frac{ C t^{3/2} \log^{9/2} n}{\sqrt{n}} \|v \|_2 .\]
Here $C > 0$ depends only on $B$.  
\end{theorem}
\begin{remark}
More generally, Theorem \ref{thm:RVdeloc} holds in the case when the entries $a_{ij}$ of $A$ are complex-valued; see \cite[Remark 1.2]{RVdeloc} for details.  
\end{remark}

In this note, we are interested in the smallest coordinates of the eigenvectors.  For comparison, a random vector $v$ uniformly distributed on the unit sphere has the following bounds for its smallest coordinates.  

\begin{proposition}[Smallest coordinates of a vector uniformly distributed on the unit sphere] \label{prop:uniform-vector}
Let $v$ be uniformly distributed on the unit sphere in either $\mathbb{R}^n$ or $\mathbb{C}^n$.  
\begin{itemize}
\item (Real case) There exists constants $C, c > 0$ such that if $v$ is uniformly distributed on the unit sphere in $\mathbb{R}^n$, then for any integer $1 \leq m \leq cn$ 
\begin{equation} \label{eq:nogaps}
	\|v_I \|_2 \geq \frac{C}{\log^c n} \left( \frac{m}{n} \right)^{3/2} \quad \text{ for all } I \subset [n], |I| \geq m
\end{equation}
with probability $1-o(1)$.  
\item (Complex case) There exists constants $C, c > 0$ such that if $v$ is uniformly distributed on the unit sphere in $\mathbb{C}^n$, then for any integer $1 \leq m \leq cn$ 
\begin{equation} \label{eq:nogaps2}
	\|v_I \|_2 \geq \frac{C}{\log^c n} \frac{m}{n } \quad \text{ for all } I \subset [n], |I| \geq m
\end{equation}
with probability $1-o(1)$.  
\end{itemize}
\end{proposition}

The bounds in \eqref{eq:nogaps} and \eqref{eq:nogaps2} show that no set of $m$ coordinates of $v$ can have too little mass.  This rules out ``gaps'' in how the mass of $v$ is spread amongst the coordinates (or as described in \cite{Rlecturenotes}, this shows that $v$ lacks ``almost empty zones'').  This phenomenon was named \emph{no-gaps delocalization} by Rudelson and Vershynin \cite{RVnogaps}.  The bounds in \eqref{eq:nogaps} and \eqref{eq:nogaps2} are conjectured to be optimal, modulo logarithmic corrections, for all values of $m$, and have been proven sharp for a number of regimes \cite{MR1994955, OVWsurvey}.  

Importantly, we emphasis the very different behavior displayed in Proposition \ref{prop:uniform-vector} between a vector $v$ uniformly distributed on the unit sphere in $\mathbb{R}^n$ compared to the unit sphere in $\mathbb{C}^n$.  This can be explained in a number of ways.  In either case, the vector $v$ has the same distribution as $g / \|g\|_2$, where $g$ is the standard real or complex Gaussian vector in $\mathbb{R}^n$ or $\mathbb{C}^n$.  It follows that, $\|g\|_2 = \Theta( \sqrt{n} )$ with probability at least $1 - C e^{-c n}$; see for example, \cite[Lemma 1]{LM}.  Here $C, c > 0$ are constants which may change from one occurrence to the next.  In addition, for all $\eps > 0$, the coordinates of $g$ satisfy 
\[ \Prob \left( |g_i| \leq \eps \right) \leq C \eps \]
in the real case and 
\[ \Prob \left( |g_i| \leq \eps \right) \leq C \eps^2 \]
in the complex case.  These bounds imply the following bounds for the coordinates of a vector $v$ uniformly distributed on the unit sphere:
\begin{equation} \label{eq:vibnd}
	\Prob \left( |v_i| \leq \frac{\eps}{\sqrt{n}} \right) \leq C \eps + C e^{-c n}
\end{equation}
in the real case and
\begin{equation} \label{eq:vibnd2}
	\Prob \left( |v_i| \leq \frac{\eps}{\sqrt{n}} \right) \leq C \eps^2 + C e^{-c n}
\end{equation} 
for the complex case.  Importantly, the difference between $\eps$ appearing on the right-hand side of \eqref{eq:vibnd} and $\eps^2$ on the right-hand side of \eqref{eq:vibnd2} leads to the differing behaviors seen in Proposition \ref{prop:uniform-vector}.  Indeed, by the union bound, \eqref{eq:vibnd} and \eqref{eq:vibnd2} can easily be used to deduce bounds for the smallest coordinate of $v$:
\[ \Prob \left( \min_{1 \leq i \leq n}  |v_i| \leq \frac{\eps}{n^{3/2}} \right) \leq C \eps + C n e^{-c n} \]
in the real case and
\[ \Prob \left( \min_{1 \leq i \leq n} |v_i| \leq \frac{\eps}{n} \right) \leq C \eps^2 + C ne^{-c n} \]
for the complex case.  These last two bounds agree with Proposition \ref{prop:uniform-vector} for the case $m=1$.

For eigenvectors of independent-entry matrices, Rudelson and Vershynin \cite{RVnogaps} proved the following analogue of \eqref{eq:nogaps}.  

\begin{theorem}[Theorem 1.5, \cite{RVnogaps}] \label{thm:RVnogaps}
Let $A$ be an $n \times n$ random matrix whose entries are iid copies of the real-valued random variable $\xi$, which satisfies 
\[ \sup_{u \in \mathbb{R}} \P(|\xi - u| \leq 1) \leq 1 - p, \quad \P(|\xi| > K) \leq p/2 \]
for some $K, p > 0$.  Choose $M \geq 1$ such that the event $\{ \|A \| \leq M \sqrt{n} \}$ holds with probability at least $1/2$.  Let $\eps \geq 1/n$ and $s \geq c_1 \eps^{-7/6} n^{-1/6} + e^{-c_2 /\sqrt{\eps}}$.  Then, conditionally on the event $\{\|A \| \leq M \sqrt{n} \}$, the following holds with probability at least $1 - (c_2 s)^{\eps n}$.  Every eigenvector $v$ of $A$ satisfies 
\[ \| v_I \|_2 \geq (\eps s)^6 \|v\|_2 \quad \text{ for all } I \subset [n], |I| \geq \eps n. \]
Here $c_1, c_2, c_3$ depend on $p$, $k$, and $M$.  
\end{theorem}

\begin{remark}
We have stated Theorem \ref{thm:RVnogaps} for real iid random matrices, but the results in \cite{RVnogaps} also extend to the case where the $(i,j)$-entry depends on the $(j,i)$-entry as well as the case when the entries are complex-valued.  We refer the reader to \cite[Section 1]{RVnogaps} for details.  
\end{remark}

In view of numerical simulations and heuristic arguments coming from \eqref{eq:vibnd} and \eqref{eq:vibnd2}, the bounds in Theorem \ref{thm:RVnogaps} appear to be suboptimal.  In this article, we improve the bounds in Theorem \ref{thm:RVnogaps} for random matrices with genuinely complex entries.  

\begin{definition}
Following \cite{luh2018complex}, we say an $N \times n$ random matrix $A$ is \emph{genuinely complex} if the entries of $A$ are independent and\footnote{We use $\sqrt{-1}$ to denote the imaginary unit and reserve $i$ as an index.  See Section \ref{sec:notation} for a complete description of our notation.}  $a_{ij} = \xi_{ij} + \sqrt{-1} \xi'_{ij}$ where $\xi_{ij}$ and $\xi'_{ij}$ are independent real random variables with mean zero, unit variance, and subgaussian moment bounded by $B$. 
\end{definition}


Eigenvectors of random matrices have been widely studied in the mathematics and physics literature.  We refer the reader to \cite{MR3494576,PhysRevLett.81.3367,MR2839984,MR3660521,MR3851824,MR3433288,MR2558268,MR3039394,MR3256861,MR2330979,MR2782201,MR3039372,MR3164751,MR3227063,MR3129806,MR3606475,MR3690289,MR3025715,MR2537522,MR3034787,MR3449389,MR2525652,MR1003705,MR1062064,MR2930379,MR3470349,MR3755583,MR2669449,MR2784665,MR3183577,MR2782623,MR2846669,MR3085669,MR2981427,BNST, BGZ, BD, CR, MC, OVWsurvey, RVdeloc, RVnogaps, TVcovariance, VuWangProjections} and references therein for many results concerning various models of random matrices.  The majority of these results apply to eigenvectors of Hermitian or real symmetric random matrices.  Significantly less appears to be known for independent-entry matrices.  
In the case of the complex Ginibre ensemble, where the entries are iid standard complex Gaussians, a number of results have described the asymptotic correlations and overlaps between eigenvectors.  Important contributions in this line of research were made by Chalker and Mehlig \cite{PhysRevLett.81.3367,MC} with significant improvements and generalizations being made recently by Fyodorov \cite{MR3851824} as well as by Bourgade and Dubach \cite{BD}.  Other recent results include \cite{BNST, BGZ, CR} and references therein, and there still appears to be significant work to be done in this area.  


Shortly after this paper appeared on the arXiv, an improved version of Theorem \ref{thm:RVnogaps} was proved by Lytova and Tikhomirov \cite{LT}.  For larger values of $m$, the results in \cite{LT} achieve the optimal bounds depicted in Proposition \ref{prop:uniform-vector}.  In particular, the results from \cite{LT} improve  upon our bounds when $m \geq \log^C n$ for some constant $C > 0$.  The techniques used by Lytova and Tikhomirov are significantly different than those employed in this paper.  In \cite{LT}, a geometric approach is taken, which utilizes test projections and involves studying random ellipsoids generated by projections of independent vectors.  Compared to  \cite{LT}, the main results in the present paper hold with higher probability and  include bounds for the cases when $1 \leq m \leq \log^C n$.

\section{Main results} \label{sec:main}

Our main results hold for random matrices with genuinely complex entries as well as random matrices with real entries.  In this section, we also discuss an application of our methods to normal vectors of random hyperplanes.  We continue to use the notation introduced above: for a vector $v = (v_i)_{i=1}^n \in \mathbb{C}^n$, we let $\|v\|_2$ denote the $\ell^2$-norm of $v$.  In addition, for $I \subset [n] := \{1, \ldots, n\}$, we let $v_I$ denote the $|I|$-vector $v_I = (v_i)_{i \in I}$, where, $|I|$ denotes the cardinality of the finite set $I$.  Recall that $\sqrt{-1}$ denotes the imaginary unit.

\subsection{Results for eigenvectors of genuinely complex matrices} \label{sec:main:eigenvectors}

Our first main result improves upon Theorem \ref{thm:RVnogaps} for large enough values of $m$.  

\begin{theorem} \label{thm:eigenvector}
Assume $A$ is an $n \times n$ genuinely complex random matrix.  Then there exist constants $C, c, c' > 0$ such that for every $t \geq e^{-\log^2 n}$ and $\log^2 n \leq m \leq c'n$, with probability at least $1 - (Ct)^m - Ce^{-c'n}$, every eigenvector $v$ of $A$ satisfies
\[ \|v_I \|_2 \geq c \sqrt{t} \left( \frac{m}{n} \right)^{3/2} \|v \|_2 \text{ for all } I \subset[n] \text{ with } |I| \geq m. \]
Here $C, c, c'$ depend only on the subgaussian moment bound $B$.  
\end{theorem}

For smaller values of $m$, we have the following bound.  

\begin{theorem} \label{thm:eigenvector-small}
Assume $A$ is an $n \times n$ genuinely complex random matrix.  Then there exist constants $C, c, c' > 0$ such that for every $t \geq e^{-c' n}$ and $1 \leq m \leq \log^2 n$, with probability at least $1 - (Ct)^m - C e^{-c'n}$, every eigenvector $v$ of $A$ satisfies
\[ \|v_I \|_2 \geq c \frac{\sqrt{t}}{\log^{2} n} \left( \frac{m}{n} \right)^{3/2 + 1/m} \|v \|_2 \text{ for all } I \subset[n] \text{ with } |I| \geq m. \]
Here $C, c, c'$ depend only on the subgaussian moment bound $B$. 
\end{theorem}

As a corollary, we immediately obtain the following in the case that $m = 1$.  

\begin{corollary} \label{cor:mincoord}
Assume $A$ is an $n \times n$ genuinely complex random matrix.  Then there exist constants $C, c, c' > 0$ such that for every $t \geq e^{-c' n}$, with probability at least $1 - Ct$, every eigenvector $v$ of $A$ satisfies
\[ |v_i| \geq c \sqrt{t} \frac{1}{n^{5/2} \log^2 n} \|v \|_2 \text{ for all } 1 \leq i \leq n.  \]
Here $C, c, c'$ depend only on the subgaussian moment bound $B$. 
\end{corollary}

Corollary \ref{cor:mincoord} implies that with probability at least $1 - C e^{-c'n}$, every coordinate of every eigenvector is nonzero.  In particular, this implies that, with the same probability, each eigenspace of $A$ has dimension one.  Indeed, if $A$ has an eigenspace of dimension greater than one, then this eigenspace must have a non-trivial intersection with the orthogonal complement of the space spanned by $e_i$, where $e_1, \ldots, e_n$ are the standard basis elements in $\mathbb{C}^n$.

\subsection{Results for eigenvectors of matrices with real entries} 

In this subsection, we consider eigenvectors of real matrices.  Our first result is the analogue of Theorem \ref{thm:eigenvector} for the eigenvectors of $A$ corresponding to real eigenvalues.  When the entries of $A$ are iid real standard normal random variables, the number of real eigenvalues was studied in \cite{MR1231689}.  The existence of real eigenvalues for random matrices with non-normal entries was established more recently in \cite{TVnonhermitian}, under the assumption the entries match the first four moments of the standard Gaussian distribution.  

\begin{theorem} \label{thm:eigenvectorreal}
Assume $A$ is an $n \times n$ real random matrix whose entries are independent copies of a mean zero subgaussian random variable with unit variance.  Then there exist constants $C, c, c' > 0$ such that for every $t \geq e^{-c' n}$ and $\log^2 n \leq m \leq c'n$, with probability at least $1 - (Ct)^m - Ce^{-c'n}$, every eigenvector $v \in \mathbb{R}^n$ of $A$ corresponding to a real eigenvalue satisfies
\[ \|v_I \|_2 \geq c {t} \left( \frac{m}{n} \right)^{2} \|v \|_2 \text{ for all } I \subset[n] \text{ with } |I| \geq m. \]
Here $C, c, c'$ depend only on the subgaussian moment of the entries.  
\end{theorem}

For smaller values of $m$, we have the following bound.  

\begin{theorem} \label{thm:eigenvector-smallreal}
Assume $A$ is an $n \times n$ real random matrix whose entries are independent copies of a mean zero subgaussian random variable with unit variance.  Then there exist constants $C, c, c' > 0$ such that for every $t \geq e^{-c' n}$ and $1 \leq m \leq \log^2 n$, with probability at least $1 - (Ct)^m - Ce^{-c'n}$, every eigenvector $v \in \mathbb{R}^n$ of $A$ corresponding to a real eigenvalue satisfies
\[ \|v_I \|_2 \geq c {t} \left( \frac{m}{n} \right)^{2 + 1/m} \|v \|_2 \text{ for all } I \subset[n] \text{ with } |I| \geq m. \]
Here $C, c, c'$ depend only on the subgaussian moment of the entries.  
\end{theorem}

\subsection{Normal vectors to random hyperplanes} \label{sec:main:normal}

As an application of our methods, we now consider delocalization bounds for normal vectors to random hyperplanes.  
Let $A$ be an $(n-1) \times n$ independent-entry random matrix.  As this matrix is ill-conditioned, there exists at least one unit vector $v$ so that $Av = 0$.  Stated another way, this means that there is at least one unit vector that is orthogonal to the rows of $A$.  In fact, under very general conditions on the entries, $A$ has rank $n-1$.  In this case, $v$ is uniquely determined up to a phase.  
Nguyen and Vu studied the normal vector $v$ when the entries of $A$ are centered iid subgaussian random variables \cite{NV}. 

In this section, we extend the results in \cite{NV} to include some additional delocalization properties for the normal vector $v$.  Intuitively, one expects $v$ to behave like a random vector uniformly distributed on the unit sphere.  In fact, in the case when $A$ has standardized Gaussian entries this is precisely the distribution of $v$.  

We begin by considering the case when $m$ is proportional to $n$.  Let us introduce the following notation.  Let $F$ be the cumulative distribution function of the $\chi^2$-distribution with two degrees of freedom. Following the notation in \cite{CHM}, let $Q$ denote the quantile function of $F$.  That is, 
\begin{equation} \label{eq:def:Q}
	Q(s) := \inf \{ x \in \mathbb{R} : F(x) \geq s \}, \quad 0 < s \leq 1, \quad Q(0) := \lim_{s \searrow 0} Q(s). 
\end{equation}
Define 
\begin{equation} \label{eq:def:H}
	H(s) := - Q(1-s), \quad 0\le s <1. 
\end{equation}

\begin{theorem} \label{thm:nogapsasym}
Suppose $\xi$ and $\xi'$ are iid real subgaussian random variables with mean zero and unit variance.  Let $A$ be an $(n-1) \times n$ iid matrix whose entries are iid copies of $\xi + \sqrt{-1} \xi'$, and let $v$ be any unit vector which satisfies $Av = 0$.  Then, for any fixed $1 > \delta > 0$, 
$$ \max_{I\subset [n]: |I|=\lfloor\delta n \rfloor} \|v_I\|^2_2 \longrightarrow -\int_{0}^{\delta} H(u) \, du  $$
and
$$ \min_{I\subset [n]: |I|=\lfloor\delta n \rfloor} \|v_I\|^2_2 \longrightarrow -\int_{1-\delta}^1 H(u) \, du  $$
in probability as $n \to \infty$, where $H$ is defined in \eqref{eq:def:H}.
\end{theorem}

\begin{remark} \label{rem:delta-mass-order}
Following \cite{OVWsurvey}, one can show that, as $\delta$ tends to zero, 
$$ -\int_{0}^{\delta} H(u)~ du = \Theta \left(  \delta \log (\delta^{-1})  \right) $$
and
$$ -\int_{1-\delta}^{1} H(u)~ du = \Theta (\delta^2). $$
In other words, Theorem \ref{thm:nogapsasym} implies that the smallest $\delta n$ coordinates of an eigenvector contribute only $\Theta(\delta^2)$ fraction of the mass, which matches the bounds from Proposition \ref{prop:uniform-vector}.  
\end{remark}

Our next results are the analogues of Theorems \ref{thm:eigenvector} and \ref{thm:eigenvector-small} for the normal vector.    

\begin{theorem} \label{thm:nogaps}
Assume $A$ is an $(n-1) \times n$ genuinely complex matrix.  Then there exist constants $C, c, c' > 0$ such that for every $t \geq e^{-\log^2 n}$ and $\log^2 n \leq m \leq c'n$, with probability at least $1 - (Ct)^m - Ce^{-c'n}$, every nonzero vector $v$ orthogonal to the rows of $A$ satisfies 
\[ \|v_I \|_2 \geq c \sqrt{t} \left( \frac{m}{n} \right)^{3/2} \|v \|_2 \text{ for all } I \subset[n] \text{ with } |I| \geq m. \]
Here $C, c, c'$ depend only on the subgaussian moment bound $B$. 
\end{theorem}

\begin{theorem} \label{thm:nogaps-small}
Assume $A$ is an $(n-1) \times n$ genuinely complex matrix.  Then there exist constants $C, c, c' > 0$ such that for every $t > 0$ and $1 \leq m \leq \log^2 n$, with probability at least $1 - (Ct)^m - C e^{-c'n}$, every nonzero vector $v$ orthogonal to the rows of $A$ satisfies 
\[ \|v_I \|_2 \geq c \frac{\sqrt{t}}{\log n} \left( \frac{m}{n} \right)^{3/2} \|v \|_2 \text{ for all } I \subset[n] \text{ with } |I| \geq m. \]
Here $C, c, c'$ depend only on the subgaussian moment bound $B$. 
\end{theorem}

\begin{remark}
More generally, Theorems \ref{thm:nogaps} and \ref{thm:nogaps-small} can be extended to cases where $A$ is an $(n-k) \times n$ matrix and $m \geq k$ using the same methods, but the lower bound for $\|v_I \|_2$ in these cases is substantially more cumbersome to notate.      
\end{remark}

\subsection{Outline of the paper}

The paper is organized as follows.  In Section \ref{sec:outline}, we give an overview of our argument by showing how delocalization properties for the eigenvectors of the square matrix $A$ can be reduced to questions concerning the least singular value of rectangular sub-matrices of $A$.  Similar reductions have been utilized before, and our arguments in this section follow closely those in \cite{RVnogaps}.  We establish a bound for the least singular value of genuinely complex rectangular random matrices in Section \ref{sec:proof:rectangular}.  This bound is based on a similar bound for genuinely complex square random matrices established by the first author \cite{luh2018complex}.  The main results in Section \ref{sec:main} are proven in Section \ref{sec:proof}.  The proofs of Propositions \ref{prop:largest-coord} and \ref{prop:uniform-vector} are presented in Appendix \ref{sec:uniform-vector}.

\subsection{Notation} \label{sec:notation}
We use asymptotic notation (such as $O,o$) under the assumption that $n \to \infty$.  We use $X = O(Y)$, $Y = \Omega(X)$, $X \ll Y$, or $Y \gg X$ to denote the estimate $|X| \leq C Y$ for some constant $C > 0$ independent of $n$ and all $n \geq C$.  If $C$ depends on another parameter, e.g., $C = C_k$, we will indicate this by subscripts, e.g., $X = O_k(Y)$ or $X \ll_k Y$.  We write $X = \Theta(Y)$ if $X \ll Y \ll X$.  We write $X = o(Y)$ if $|X| \leq c_n Y$ for some $c_n$ that tends to zero as $n \to \infty$.  

$|S|$ denotes the cardinality of the finite set $S$.  We use $\sqrt{-1}$ to denote the imaginary unit and reserve $i$ as an index.  $[n]$ denotes the discrete interval $\{1, \ldots, n\}$.  

We use $S_{\mathbb{R}}^{n-1}$ and $S_{\mathbb{C}}^{n-1}$ to denote the unit spheres in $\mathbb{R}^n$ and $\mathbb{C}^n$, respectively.  For a vector $x = (x_i)_{i=1}^n \in \mathbb{C}^n$, we let $\|x\|_2$ denote the $\ell^2$-norm of $x$.  In addition, for $I \subset [n] := \{1, \ldots, n\}$, we let $x_I$ denote the $|I|$-vector $x_I = (x_i)_{i \in I}$.  Similarly, for an $N \times n$ matrix $A = (a_{ij})_{i \in [N], j \in [n]}$ and a subset $J \subset [n]$, we let $A_J$ denote the $N \times |J|$ matrix $A_J = (a_{ij})_{i \in [N], j \in J}$.    We let $\mathcal{M}_{N \times n}^\mathbb{F}$ denote the set of $N \times n$ matrices over the field $\mathbb{F}$ (here, $\mathbb{F}$ is either $\mathbb{R}$ or $\mathbb{C}$).  

Recall that the singular values of a matrix $A$ are the square roots of the eigenvalues of $A^\ast A$.  For an $N \times n$ matrix $A$, we let $s_1(A) \geq \cdots \geq s_n(A)$ denote the ordered singular values of $A$.  Of particular importance are the largest and smallest singular values:
\[ s_1(A) = \max_{\|x\|_2 = 1} \|Ax \|_2, \qquad s_n(A) = \min_{ \|x\|_2 = 1} \|Ax \|_2. \]
We will let $\|A\|$ denote the spectral norm of $A$, i.e., $\|A\| = s_1(A)$.  For convenience, we will often let $s_{\min}(A)$ denote the smallest singular value of $A$.

\section{Outline of the argument}  \label{sec:outline}

\subsection{Reduction of delocalization to invertibility}

For an $n \times n$ matrix $A$, introduce the localization event
\[ \loc(A, m, \delta) = \{ \exists \text{ eigenvector } v \in S^{n-1}_{\mathbb{C}} \text{ of } A, \exists I \subset [n], |I| = m, \|v_I\|_2 < \delta \}. \]
Intuitively, $\loc(A, m, \delta)$ captures the event that $A$ has an eigenvector which has a subset of $m$ coordinates which carry a disproportionately small proportion of the mass.  

We will also extend this notion to rectangular matrices, but first we fix some notation.  If $A$ is a square matrix and $\lambda \in \mathbb{C}$, then $A - \lambda$ denotes the matrix $A - \lambda I$, where $I$ is the identity matrix.  Similarly, if $A$ is a rectangular matrix, we define $A - \lambda$ to be the $N \times n$ matrix with entries $A_{ij} - \lambda \delta_{ij}$, where $\delta_{ij}$ is the Kronecker delta.  

For an $N \times n$ matrix $A$ and $\lambda_0 \in \mathbb{C}$, we define the localization event
\[ \loc_{\lambda_0}(A, M, m, \delta) = \{ \exists v \in S^{n-1}_{\mathbb{C}}, \|(A - \lambda_0) v\|_2 \leq \delta M \sqrt{n}, \exists I \subset [n], |I| = m, \|v_I\|_2 < \delta \}. \]
In the case when $A$ is square, $\loc_{\lambda_0}(A, M, m, \delta)$ is the event that an approximate eigenvector $v$ (with approximate eigenvalue $\lambda_0$) is localized.  For Theorems \ref{thm:nogaps} and \ref{thm:nogaps-small} it is important that we allow this event to also apply to rectangular matrices.  

The following three propositions are based on \cite[Proposition 4.1]{RVnogaps} and show that the study of the localization events defined above can be reduced to a question involving the least singular value of the random matrix $A$.  

\begin{proposition}[Reduction of delocalization to invertibility for approximate eigenvectors] \label{prop:delocgeneral}
Let $A$ be an $N \times n$ random matrix with arbitrary distribution.  Let $M \geq 1$, $\delta \in (0,1/2)$, $p_0 \in (0,1)$, $m \in [n]$, and $\lambda_0 \in \mathbb{C}$ with $|\lambda_0| \leq M \sqrt{n}$.  Assume that for any set $I \subset [n]$ with $|I| = m$, we have
\begin{equation} \label{eq:p_0def1}
	\Prob( s_{\min}( (A - \lambda_0)_{I^c}) \leq 6 \delta M \sqrt{n} \text{ and } \|A \| \leq M \sqrt{n} ) \leq p_0 . 
\end{equation}
Then
\[ \Prob( \loc_{\lambda_0}(A, M, m, \delta) \text{ and } \|A \| \leq M \sqrt{n} ) \leq \left( \frac{n e}{m} \right)^m p_0. \]
\end{proposition}
\begin{proof}
Assume $\|A \| \leq M \sqrt{n}$ and the localization event $\loc_{\lambda_0}(A, M, m, \delta)$ holds.  Then there exists a unit vector $v$ and an index set $I \subset [n]$ with $|I| = m$ such that
\begin{equation} \label{eq:deltabnd}
	 \|(A - \lambda_0)v \|_2 \leq \delta M \sqrt{n} \text{ and } \|v_I\|_2 < \delta. 
\end{equation}
We decompose the vector $v$ as $v = v_I + v_{I^c}$ to obtain
\[ (A - \lambda_0)v = (A - \lambda_0)_I v_I + (A- \lambda_0)_{I^c} v_{I^c}. \]
Using \eqref{eq:deltabnd}, we find
\[ \|(A- \lambda_0)_{I^c} v_{I^c} \|_2 \leq 3 \delta M \sqrt{n} \]
and so
\[ s_{\min}((A - \lambda_0)_{I^c}) \|v_{I^c}\|_2 \leq 3 \delta M \sqrt{n}. \]
Since $\|v_I\|_2 < \delta \leq 1/2$, we obtain $\|v_{I^c} \|_2 \geq 1/2$, and hence
\begin{equation} \label{eq:sminrequire}
	s_{\min}((A - \lambda_0)_{I^c}) \leq 6 \delta M \sqrt{n}. 
\end{equation}

In other words, we have shown that the events $\|A \| \leq M \sqrt{n}$ and $\loc_{\lambda_0}(A, M, m, \delta)$ imply the existence of a subset $I \subset [n]$ with $|I| = m$ such that \eqref{eq:sminrequire} holds.  Applying the union bound and \eqref{eq:p_0def1}, we conclude that 
\[ \Prob( \loc_{\lambda_0}(A, m, \delta) \text{ and } \|A \| \leq M \sqrt{n} ) \leq \binom{n}{m} p_0 \leq \left( \frac{n e }{m} \right)^m p_0, \]
as desired.  
\end{proof}

\begin{proposition}[Reduction of delocalization to invertibility for eigenvectors]  \label{prop:deloc}
Let $A$ be an $n \times n$ random matrix with arbitrary distribution.  Let $M \geq 1$, $\delta \in (0,1/2)$, $p_0 \in (0,1)$, and $m \in [n]$.  Assume that for any set $I \subset [n]$ with $|I| = m$ and any $\lambda_0 \in \mathbb{C}$ with $|\lambda_0| \leq M \sqrt{n}$, we have
\begin{equation} \label{eq:deloc:p_0}
	\Prob( s_{\min}( (A - \lambda_0)_{I^c}) \leq 6 \delta M \sqrt{n} \text{ and } \|A \| \leq M \sqrt{n} ) \leq p_0. 
\end{equation}
Then
\[ \Prob( \loc(A, m, \delta) \text{ and }  \|A \| \leq M \sqrt{n} ) \leq \frac{9}{\delta^2} \left( \frac{n e}{m} \right)^m p_0. \]
\end{proposition}
\begin{proof}
Suppose $\|A \| \leq M \sqrt{n}$ and the localization event $\loc(A, m, \delta)$ holds.  Then there exists an eigenvector $v \in S^{n-1}_{\mathbb{C}}$ of $A$ and an index set $I \subset [n]$ with $|I| = m$ such that $\|v_I \|_2 < \delta$.  Let $\lambda$ be the eigenvalue of $A$ corresponding to the eigenvector $v$.  Then $|\lambda| \leq \|A \| \leq M \sqrt{n}$.  Let $\mathcal{N}$ be a $(\delta M \sqrt{n})$-net of the disc $\{z \in \mathbb{C} : |z| \leq M \sqrt{n} \}$.  A simple volume argument shows that one can construct the net $\mathcal{N}$ to have cardinality 
\begin{equation} \label{eq:Ncardbnd}
	|\mathcal{N}| \leq \frac{9}{\delta^2}. 
\end{equation}
Choose $\lambda_0 \in \mathcal{N}$ such that $|\lambda  - \lambda_0| \leq \delta M \sqrt{n}$.  Then the eigenvalue-eigenvector equation $Av = \lambda v$ implies
\[ (A - \lambda_0) v = (\lambda - \lambda_0) v, \]
and hence
\[ \| (A - \lambda_0)v \|_2 \leq | \lambda - \lambda_0| \leq \delta M \sqrt{n}. \]

To summarize, we have shown that the events $\|A \| \leq M \sqrt{n}$ and $\loc(A, m, \delta)$ imply the existence of $\lambda_0 \in \mathcal{N}$ such that $\loc_{\lambda_0}(A, M, m, \delta)$ holds.  We conclude from Proposition \ref{prop:delocgeneral} and the union bound that
\[ \Prob( \loc(A, m, \delta) \text{ and }  \|A \| \leq M \sqrt{n} ) \leq |\mathcal{N}| \left( \frac{n e}{m} \right)^m p_0. \]
Combining this bound with \eqref{eq:Ncardbnd} completes the proof.  
\end{proof}

To work with the eigenvectors of real matrices, we also require the following event: 
\[ \loc_{\mathbb{R}}(A,m,\delta) := \{ \exists \text{ eigenvector } v \in S_{\mathbb{R}}^{n-1} \text{ of } A, \exists I \subset [n], |I| = m, \|v_I\|_2 < \delta \}. \]
In this case, we have the following analogue of Proposition \ref{prop:deloc}.

\begin{proposition}[Reduction of delocalization to invertibility for real eigenvectors] \label{prop:delocreal}
Let $A$ be an $n \times n$ real random matrix with arbitrary distribution.  Let $M \geq 1$, $\delta \in (0,1/2)$, $p_0 \in (0,1)$, and $m \in [n]$.  Assume that for any set $I \subset [n]$ with $|I| = m$ and any $\lambda_0 \in \mathbb{R}$ with $|\lambda_0| \leq M \sqrt{n}$, we have
\begin{equation} \label{eq:deloc:p_02}
	\Prob( s_{\min}( (A - \lambda_0)_{I^c}) \leq 6 \delta M \sqrt{n} \text{ and } \|A \| \leq M \sqrt{n} ) \leq p_0. 
\end{equation}
Then
\[ \Prob( \loc_{\mathbb{R}}(A, m, \delta) \text{ and }  \|A \| \leq M \sqrt{n} ) \leq \frac{3}{\delta} \left( \frac{n e}{m} \right)^m p_0. \]
\end{proposition}
\begin{proof}
The proof follows a similar argument as the proof of Proposition \ref{prop:deloc}.  Suppose $\|A \| \leq M \sqrt{n}$ and the localization event $\loc_{\mathbb{R}}(A, m, \delta)$ holds.  Then there exists an eigenvector $v \in S^{n-1}_{\mathbb{R}}$ of $A$ and an index set $I \subset [n]$ with $|I| = m$ such that $\|v_I \|_2 < \delta$.  Let $\lambda$ be the eigenvalue of $A$ corresponding to the eigenvector $v$.  Since the matrix $A$ has real entries, the eigenvalue $\lambda$ must be real.  In addition, $|\lambda| \leq \|A \| \leq M \sqrt{n}$.  Let $\mathcal{N}$ be a $(\delta M \sqrt{n})$-net of the real interval $[-M\sqrt{n}, M \sqrt{n}]$.  A simple volume argument shows that one can construct the net $\mathcal{N}$ to have cardinality 
\begin{equation} \label{eq:Ncardbnd2}
	|\mathcal{N}| \leq \frac{3}{\delta}. 
\end{equation}
Choose $\lambda_0 \in \mathcal{N}$ such that $|\lambda  - \lambda_0| \leq \delta M \sqrt{n}$.  Then the eigenvalue-eigenvector equation $Av = \lambda v$ implies
\[ (A - \lambda_0) v = (\lambda - \lambda_0) v, \]
and hence
\[ \| (A - \lambda_0)v \|_2 \leq | \lambda - \lambda_0| \leq \delta M \sqrt{n}. \]

To summarize, we have shown that the events $\|A \| \leq M \sqrt{n}$ and $\loc_{\mathbb{R}}(A, m, \delta)$ imply the existence of $\lambda_0 \in \mathcal{N}$ such that $\loc_{\lambda_0}(A, M, m, \delta)$ holds.  We conclude from Proposition \ref{prop:delocgeneral} and the union bound that
\[ \Prob( \loc(A, m, \delta) \text{ and }  \|A \| \leq M \sqrt{n} ) \leq |\mathcal{N}| \left( \frac{n e}{m} \right)^m p_0. \]
Combining this bound with \eqref{eq:Ncardbnd2} completes the proof.  
\end{proof}

\subsection{Least singular value of rectangular matrices}
In order to apply Propositions \ref{prop:delocgeneral}, \ref{prop:deloc}, and \ref{prop:delocreal}, we will need bounds on the least singular value of genuinely complex random matrices.  These bounds are the key technical achievement of this paper.  Indeed, the results below provide an analogue of the main result in \cite{RVRectangle} for genuinely complex random matrices.  
\begin{theorem}             \label{thm:rectangular}
  Let $A$ be an $N \times n$ random genuinely complex matrix, $N \ge n$, and $\lambda \in \C$ with $|\lambda| \leq M \sqrt{N}$ for $M \geq 1$.
  Then, for every $\e > 0$, we have
  \begin{equation}\label{eq rectangular subgaussian}
    \P \Big( s_n(A - \lambda) \le \e \big (\sqrt{N} - \sqrt{n-1} \big ) \Big)
    \le (C \e)^{2(N-n+1)-1} + e^{-cN}
  \end{equation}
  where $C, c > 0$ depend (polynomially)
  only on the subgaussian moment $B$ and $M$.
\end{theorem}
\begin{remark}
Note that in \cite{RVRectangle}, the upperbound in (\ref{eq rectangular subgaussian}) for real random matrices and $\lambda = 0$ is of the form $(C \e)^{N -n+1} + e^{-cN}$.  Essentially, we have replaced this $\e$ in the real case with $\e^2$ in the genuinely complex case.  The right hand side is near optimal up to a factor of $\eps$.  
\end{remark}

By slightly altering the proof of Theorem \ref{thm:rectangular}, we are able to prove a bound that is more effective in the regime where $N -n$ is small.  
\begin{theorem}             \label{thm:nearlysquare}
  Let $A$ be an $N \times n$ random genuinely complex matrix, $N \ge n \ge N - T$, for some integer $T$. Consider $\lambda \in \C$ with $|\lambda| \leq M \sqrt{N}$ for $M \geq 1$.
  Then, for every $\e > 0$, we have
  \begin{equation}\label{eq nearlysquare}
    \P \Big( s_n(A - \lambda) \le \e \big (\sqrt{N} - \sqrt{n-1} \big ) \Big)
    \le (C \sqrt{T} \e)^{2(N-n+1)} + e^{-cN}
  \end{equation}
  where $C, c > 0$ depend (polynomially)
  only on the subgaussian moment $B$ and $M$.
\end{theorem}

\begin{remark}
Note that when $T = O(1)$, we recover the optimal bound.
\end{remark}

\section{Proof of Theorem \ref{thm:rectangular}} \label{sec:proof:rectangular}
Our proof follows \cite{RVRectangle} closely and also combines several ideas from \cite{luh2018complex}.  We mirror the notation from \cite{RVRectangle, luh2018complex} for ease of comparison.
\subsection{Preliminaries}
\subsubsection{Nets}
Consider a subset $D$ of $\C^n$, and let $\e > 0$.
Recall that an $\e$-net of $D$ is a subset $\NN \subseteq D$ such that
for every $x \in D$ one has $\dist(x,\NN) \le \e$.

The following lemma is the complex analogue of Proposition 2.1 in \cite{RVRectangle}.  The proof is identical to the real case if one identifies $\C^n$ with $\R^{2n}$.

\begin{proposition}[Nets]                     \label{prop:nets}
  Let $S$ be a subset of $S_{\C}^{n-1}$, and let $\e > 0$.
  Then there exists an $\e$-net of $S$ of
  cardinality at most
  $$
  4n \Big( 1 + \frac{2}{\e} \Big)^{2n-1}.
  $$
\end{proposition}

Using the standard net argument, one can show the following bound on the operator norm of rectangular matrices.
\begin{proposition} \label{prop:opnorm}
Let $A$ be an $N \times n$ genuinely complex random matrix, with $N \geq n$ and $\lambda \in \C$ with $|\lambda| \leq M \sqrt{N}$ for $M \geq 1$.  Then
$$
\P\big( \|A - \lambda \| > t \sqrt{N} \big) \leq e^{-c_0 t^2 N} \qquad \text{for } t \ge C_0,
$$
where $C_0, c_0 > 0$ depend only on the subgaussian moment $B$ and $M$.
\end{proposition}

\subsubsection{Converting between $\R$ and $\C$}
Following \cite{luh2018complex}, for a vector $v \in \C^n$, we denote by $\underline{v}$ its associated real vector defined to be
$$
\underline{v} := \left(\begin{array}{c}
\RR(v) \\
\II(v) 
\end{array} \right)
$$ 
and $[v]$ denote its associated $2n \times 2$ real matrix defined to be
$$
[v] := \left(\begin{array}{cc}
\RR(v) & -\II (v) \\
\II(v) & \RR (v)
\end{array} \right)
$$ 
We generalize this notion from \cite{luh2018complex} to include matrices.  For a $M \in \mathcal{M}^{\C}_{n \times m}$ matrix with
$$
M = A + i B,
$$
where $A, B \in \mathcal{M}^{\R}_{n \times m}$, we define $[M] \in \mathcal{M}^{\R}_{2n \times 2m}$ to be 
$$
[M] := \left(\begin{array}{cc}
A & -B \\
B & A
\end{array} \right).
$$ 
We record some useful consequences of these definitions below.
\begin{lemma} \label{lem:RtoC}
For $a \in \C$, $x, y \in \C^n$ and $M \in  \mathcal{M}^\C_{m \times n}$,
$$
\|x - y\|_2  = \|\underline{x} - \underline{y}\|_2,
$$
$$
\underline{a x} = [x] \underline{a},
$$
$$
\underline{M x} = [M] \underline{x}. 
$$
\end{lemma}

\subsubsection{Decomposition of the unit sphere}
In our proof of Theorem \ref{thm:rectangular}, we utilize a partition of the unit sphere due to Rudelson and Vershynin \cite{RVLittlewoodOfford}.
\begin{definition}
Let $\d, \rho \in (0,1)$.  A vector $x \in \C^n$ is \emph{sparse} if $|\supp(x)| \leq \d n$.  A vector $x \in \C^n$ is \emph{compressible} if there exists a sparse vector $y$ such that $\|x - y\|_2 \leq \rho$.  A vector $x \in \C^n$ is \emph{incompressible} if it is not compressible.  We denote the sets of compressible and incompressible vectors by $\Comp(\d, \rho)$ and $\Incomp(\d, \rho)$ respectively. 
\end{definition}

We now recall two simple results.  
\begin{lemma}[Lemma 5.3, \cite{luh2018complex}] \label{lem:spreadset}
Let $x \in Incomp(\d, \rho)$.  Then there exists a set $\sigma \subseteq [n]$ of cardinality $|\sigma| \geq \nu_1 n$ and such that
$$
\frac{\nu_2}{\sqrt{n}} \leq |x_k| \leq \frac{\nu_3}{\sqrt{n}} \qquad \qquad \text{for all } k \in \sigma
$$
where $0 < \nu_1, \nu_2, \nu_3$ are constants depending only on $\d$ and $\rho$.  
\end{lemma}
The next lemma controls the norm of the images of compressible vectors.  We omit the proof which is a straightforward adaptation of Section 2.2 in \cite{RVLittlewoodOfford}. 

\begin{lemma} \label{lem:compressible}
Let $A$ be a $N \times n$ genuinely complex random matrix, $N \geq n/2$ and $\lambda \in \C$ with $|\lambda| \leq M \sqrt{N}$ for $M \geq 1$.  There exist $\d, \rho, c_3$ depending only on the subgaussian moment $B$ and $M$ such that
$$
\P \big( \inf_{x \in \Comp(\d, \rho)} \|(A - \lambda )x\|_2 \leq c_3 \sqrt{N} \big) \leq e^{-c_3 N}.
$$
\end{lemma}
\subsection{Small ball probability and arithmetic structure in $\R$}
At several points in the proof of Theorem \ref{thm:rectangular}, we will need quantitative control on the spread of a random variable.

\begin{definition}
The \emph{L\'evy concentration function} of a random vector $S \in \R^m$ (or $\C^m$) is defined for $\e > 0$ as
$$
\LL(S, \e) = \sup_{v \in \R^m (\text{or } \C^m)} \P(\|S-v\|_2 \leq \e).
$$
\end{definition}

Below we recount several results for real random variables.  

\begin{lemma}[Lemma 2.6, \cite{RVRectangle}] \label{lem:singlecoordinate}
Let $\xi$ be a real random variable with mean zero, unit variance, and finite fourth moment.  Then for every $\e \in (0,1)$, there exists $p \in (0,1)$ which depends only on $\e$ and on the fourth moment, and such that
$$
\LL(\xi, \e) \leq p.
$$
\end{lemma}

This rather crude bound can be significantly improved when more is known about the random variable.  In particular, a well-developed theory exists when $S = \sum_{k=1}^N a_k \xi_k$ where $a_k$ are fixed vectors and $\xi_k$ are independent random variables.  This question is the basis of Littlewood-Offord theory and the situation when $a_k$ are scalars has a long history in random matrix theory \cite{TVInverseLO, RVLittlewoodOfford}.  The fundamental observation in Littlewood-Offord theory is that the L\'evy concentration function is dependent on the additive structure of the coefficients, $a_k$.  For the scalar case, Rudelson and Vershynin \cite{RVLittlewoodOfford} defined the essential least common denominator for the vector of coefficients, $a = (a_1, \dots, a_N)$, to be
$$
\LCD_{\a, \g}(a) := \inf \left \{\theta>0 : \dist(\theta a, \mathbb{Z}^N) < \min( \gamma \|\theta a\|_2, \alpha)\right \}
$$ 
which roughly captures the length of the shortest arithmetic progression in which $a$ can be embedded.  

In \cite{RVRectangle}, Rudelson and Vershynin generalized this notion to higher dimensions.  If we now allow $a = (a_1, \dots, a_N)$ to be a sequence of vectors $a_k \in \R^m$, then we define the product of such a multi-vector $a$ and a vector $\theta \in \R^m$ as
$$
\theta \cdot a = (\< \theta, a_1 \> , \dots, \< \theta, a_N \> ) \in \R^N.
$$ 
Then we define, for $\a > 0$ and $\g \in (0,1)$,
$$
\LCD_{\a, \g}(a) := \inf \left \{ \|\theta\|_2 : \theta \in \R^m, \dist(\theta \cdot a, \Z^N) < \min(\g \|\theta \cdot a\|_2, \a) \right \}
$$

The following theorem provides a bound on the small ball probability in terms of this generalized essential least common denominator.

\begin{theorem}[Theorem 3.3, \cite{RVRectangle}] \label{thm:smallball}
Let $a \in (a_1, \dots, a_N)$ be a sequence of vectors $a_k \in \R^m$ which satisfy
$$
\sum_{k=1}^N \< a_k, x \> ^2 \geq \|x\|_2^2 \qquad \text{ for every } x \in \R^m.
$$
Let $\xi_1, \dots, \xi_N$ be independent real random variables, such that $\LL(\xi_k, 1) \leq 1-b$ for some $b>0$.  Consider the random sum $S = \sum_{k=1}^N a_k \xi_k$.  Then, for every $\a > 0$ and $\g \in (0,1)$, and for 
$$
\e \geq \frac{\sqrt{m}}{\LCD_{\a, \g} (a)},
$$
we have
$$
\LL(S, \e \sqrt{m}) \leq \left( \frac{C \e}{\gamma \sqrt{b}} \right)^m + C^m e^{-2b \a^2}.
$$
\end{theorem}

\begin{remark}
In \cite{RVRectangle}, the statement of the theorem requires identically distributed, mean zero random variables, but the proof (which begins with symmetrization anyways) can be easily altered to handle random variables with arbitrary and possibly different means.  The identical distribution requirement can also be relaxed as long as the random variables have unit variance and a uniform bound on the subgaussian moment.
\end{remark}

\subsection{Arithmetic structure in $\C$}
In \cite{luh2018complex}, the first author generalized the notion of essential least common denominator to the complex setting.  

\begin{definition}
If we let $a = (a_1, \dots, a_N)$ be a vector of complex numbers, we define the essential least common denominator of $a$ to be
$$
\LCD_{\a, \g}(a) := \inf \left \{\|\theta\|_2: \theta \in \R^2, \dist([a] \theta, \Z^{2N}) < \min(\g \|\underline{a}\theta\|, \a) \right \}
$$ 
By Lemma \ref{lem:RtoC}, an equivalent definition is
$$
\LCD_{\a, \g}(a) := \inf \left \{|\theta|: \theta \in \C, \dist(\underline{\theta a}, \Z^{2N}) < \min(\g |a \theta|, \a) \right \}
$$
\end{definition}

We extend this definition to higher dimensions below.
\begin{definition}
Let $a = (a_1, \dots, a_N)$ be a sequence of vectors $a_k \in \C^m$.  
Then we define, for $\a > 0$ and $\g \in (0,1)$,
$$
\LCD_{\a, \g}(a) := \inf \left \{ \|\theta\|_2 : \theta \in \R^{2m}, \dist([A]^T \theta, \Z^{2N}) < \min(\g \|\theta \cdot a\|_2, \a) \right \}
$$
where $A$ is the matrix with columns $a_k$.
An equivalent, more geometric, definition is the following:
$$
\LCD_{\a, \g}(a) := \inf \left \{ \|\theta\|_2 : \theta \in \C^m, \dist(\underline{\theta \cdot a}, \Z^{2N}) < \min(\g \|\theta \cdot a\|_2, \a) \right \}
$$
where we define the product of such a multi-vector $a$ and a vector $\theta \in \C^m$ as
$$
\theta \cdot a = (\< \theta, a_1 \> , \dots, \< \theta, a_N \> ) \in \C^N.
$$ 
\end{definition}

\begin{remark}
Note that the first definition makes it clear that the $\LCD$ of $N$ complex vectors can be related to the $\LCD$ of $2N$ real vectors (the $2N$ columns of $[A]$).  This allows us to use Theorem \ref{thm:smallball} in the complex setting.  
\end{remark}

\subsection{Least common denominator of incompressible vectors}
We recall a lemma from \cite{luh2018complex} which provides a lower bound on the $\LCD$ of incompressible vectors.

\begin{lemma}[Lemma 5.12, \cite{luh2018complex}] \label{lem:lowerboundLCD}
There exist constants $\gamma, \lambda > 0$ only depending on $\d$ and $\rho$ such that for any incompressible vector $x \in \Incomp(\d, \rho)$ one has for every $\a > 0$,
$$
\LCD_{\a, \g}(x) \geq \lambda \sqrt{n}.
$$
\end{lemma}

\subsection{Distance to subspaces and arithmetic structure} 
In this section, we utilize the arithmetic structure of subspaces to control the distances of random vectors to random subspaces.  It is in this section that we exploit having a genuinely complex random matrix and we gain the extra factor of $\e$ to $\e^2$ as compared to the real case.  In particular, we show the following optimal bound on the distance of a random vector to a random subspace.

\begin{theorem}[Distance to random subspace] \label{thm:distance}
Let $X$ be a vector in $\C^N$ whose coordinates are genuinely complex (but not necessarily centered) and independent.  Let $H$ be a random subspace in $\C^N$ spanned by $N-m$ genuinely complex random vectors (not necessarily centered) independent of $X$, with $0< m < \tilde{c} N$.  Then, for every $v \in \C^N$ and every $\e > 0$, we have
$$
\P(\dist(X, H + v) < \eps \sqrt{m}) \leq (C \eps)^{2m} + e^{- c N},
$$ 
where $C, c, \tilde{c} >0$ depend only on the subgaussian moment $B$.
\end{theorem}

We deduce Theorem \ref{thm:distance} via a covering argument that first requires a bound that holds for a fixed subspace and depends on the arithmetic structure of that subspace.  For $\a > 0$ and $\g \in (0,1)$, we define the \emph{essential least common denominator of a subspace} $E$ in $\C^N$ to be
$$
\LCD_{\a, \g} := \inf \left\{ \LCD_{\a, \g}(a) : a \in S(E) \right\}
$$
where $S(E)$ denotes the intersection of the unit sphere with $E$.
One can see that this is equivalent to 
$$
\LCD_{\a, \g} = \inf \left\{ \|\theta\|_2: \theta \in E, \dist(\underline{\theta}, \Z^{2N})< \min(\g \|\theta\|_2, \a) \right\}.
$$
We now combine this notion with Theorem \ref{thm:smallball} to yield the following bound on the distance.

\begin{theorem}[Distance to a general subspace] \label{thm:distancegeneral}
Let $X$ be a genuinely complex random vector (not necessarily centered) in $\C^N$.  Let $H$ be a subspace in $\C^N$ of dimension $N-m > 0$.  Then for every $v \in \C^N$, $\a > 0$, $\g \in (0,1)$, and for 
$$
\e \geq \frac{\sqrt{m}}{\LCD_{\a, \g} H^{\perp}},
$$
we have
$$
\P\big(\dist(X, H+v)< \e \sqrt{m}\big) \leq \left( \frac{C \e}{\g} \right)^{2m} + C^m e^{-c \a^2}
$$
where $C, c > 0$ depend only on the subgaussian moment $B$.
\end{theorem} 

\begin{proof}
We write $\underline{X}$ in coordinates, $\underline{X} = (\xi_1, \dots, \xi_{N}, \xi_{N+1}, \dots, \xi_{2N})$.  By Lemma \ref{lem:singlecoordinate}, each coordinate of $\underline{X}$ satisfies $\LL(\xi_k, 1/2) \leq 1-b$ for some $b > 0$ that only depends on the subgaussian moment $B$.  Thus, the random variables $\xi_k/2$ satisfy the assumptions of Theorem \ref{thm:smallball}.

Now, we convert the distance problem into a small ball probability calculation for a sum of independent vectors.  Let $P_H$ signify the orthogonal projection onto a subspace $H$.
\begin{equation}
\dist(X, H+v) = \|P_{H^\perp} (X-v) \|_2 = \left \| \sum_{k=1}^{2N} a_k \xi_k -w \right \|_2,
\end{equation}
where
$$
a_k = \left( \begin{array}{c}
\RR(P_{H^\perp} e_k) \\
\II(P_{H^\perp} e_k)
\end{array} \right), \qquad
a_{N+k} = \left( \begin{array}{c}
-\II(P_{H^\perp} e_k) \\
\RR(P_{H^\perp} e_k)
\end{array} \right), \qquad w = \underline{P_{H^\perp}v}
$$
for $1 \leq k \leq N$.  
For this sequence of vectors $a = (a_1, \dots, a_{2N})$, we have
$$
\sum_{k=1}^{2N} \< a_k, \underline{x} \> ^2 = \sum_{k=1}^N |\< P_{H^\perp} e_k, x \> |^2 = \sum_{k=1}^N |\<  e_k, x \> |^2 = \|x\|_2^2 \qquad \text{for any } x \in H^{\perp}
$$
so we can apply Theorem \ref{thm:smallball} in the space $H^\perp$ (which can be identified with $\C^m$ under a suitable isometry).  

For any $\theta \in H^\perp$ we have $\< \theta, P_{H^\perp} e_k \> = \< P_{H^\perp} \theta, e_k \> = \< \theta, e_k \> $ so by Lemma \ref{lem:RtoC},
$$
\left( \begin{array}{ccc}
 & a_k^T & \\
 & a_{N+k}^T & 
\end{array} \right)  \underline{\bar{\theta}} = [(P_{H^\perp} e_k)^T] \underline{\bar{\theta}} = \underline{(P_{H^\perp} e_k)^T \bar{\theta}} = \underline{ \< P_{H^\perp} e_k, \theta \> } = \underline{ \overline{ \< \theta, e_k \> }}.
$$
As conjugation will not alter the norm, we have
$$
\LCD_{\a, \g}(H^\perp) = \LCD_{\a, \g}(a).
$$
The result now follows from a direct application of Theorem \ref{thm:smallball}.
\end{proof}

To prove the distance bound we carry out a covering argument to exclude those possible $H^\perp$ with small $\LCD$ of a random subspace $H^\perp$.  In fact, we show that the $\LCD$ of such a subspace is typically exponentially large.

\begin{theorem}[Structure of a random subspace]\label{thm:structure}
Let $H$ be a random subspace in $\C^N$ spanned by $N-m$ genuinely complex random vectors, $1 \leq m \leq \tilde{c} N$.  Then, for $\alpha = c \sqrt{N}$, we have
$$
\P\big( \LCD_{\a, c}(H^\perp) < c \sqrt{N} e^{cN/m} \big) \leq e^{-cN},
$$
where $c \in (0,1)$ and $\tilde{c} \in (0, 1/2)$ depend only on the subgaussian moment $B$.
\end{theorem} 

For now, if we assume this result, we can complete the proof of Theorem \ref{thm:distance}.

\begin{proof}[Proof of Theorem \ref{thm:distance}]
Consider the event 
$$
\EE := \{ \LCD_{\a, c}(H^\perp) \geq c \sqrt{N} e^{cN/m} \}.
$$
By Theorem \ref{thm:structure}, $\P(\EE^c)\leq e^{-cN}$.  We now condition on a realization of $H$ in $\EE$.  By the independence of $H$ and $X$, Theorem \ref{thm:distancegeneral} applied with $\a = c \sqrt{N}$ and $\g = c$ yields
$$
\P \big( \dist(X,H)< \e \sqrt{m} | \EE \big) \leq (C_1 \e)^{2m} + C^m e^{-c_1 N}
$$
for any 
$$
\e > C_2 \sqrt{\frac{m}{N}} e^{-cN/m}.
$$
Since $m \leq \tilde{c}N$, for an appropriate choice of $\tilde{c}$ we have
$$
C_2 \sqrt{\frac{m}{N}} e^{cN/m} \leq \frac{1}{C_1} e^{-c_3 N/m} \quad \text{and } \quad C^m e^{-c_1 N} \leq e^{-c_3 N}.
$$
Thus, for every $\eps > 0$, 
$$
\P\big(\dist(X, H) < \e \sqrt{m} | \EE \big) \leq (C_1 \eps)^{2m} + 2 e^{-c_3 N} \leq (C_1 \eps)^{2m} +  e^{-c_4 N} 
$$
\end{proof}

\subsubsection{Proof of Structure Theorem \ref{thm:structure}}
Throughout the proof we assume that $N > N_0$ for some suitably large number $N_0$ which only depends on the subgaussian moment $B$.  Indeed, the assumption on $m$ implies that $N > 1/\tilde{c}$.  Thus, choosing $\tilde{c}$ small enough, we can make $N_0$ suitably large.

Let $X_1, \dots, X_{N-m}$ denote the independent random vectors that span the subspace $H$.  Consider the $(N-m) \times N$ random matrix $B$ with rows $\bar{X_k}$.  Then 
$$
H^\perp \subseteq \ker(B).
$$
Therefore, for every set $S$ in $\C^N$ we have:
\begin{equation} \label{eq:intersection}
\inf_{x \in S} \|B x\|_2 > 0 \text{ implies } H^\perp \cap S = \emptyset.
\end{equation}
This observation reduces the intersection problem to bounding the infimum of the image of $S$ under $B$.

We now show that a typical subspace is entirely contained in $\Incomp(\d, \rho)$.
\begin{lemma} \label{lem:randomisincompressible}
There exist $\d, \rho \in (0,1)$ such that 
$$
\P \big( H^\perp \cap S^{n-1} \subseteq \Incomp(\d, \rho) \big) \geq 1 - e^{-c N}.
$$
\end{lemma}
\begin{proof}
Since $N-m > (1- \tilde{c}) N$ and $\tilde{c} < 1/2$, we can apply Lemma \ref{lem:compressible} for the matrix $B$. Therefore, there exist $\d, \rho \in (0,1)$ such that 
$$
\P(\inf_{x \in \Comp(\d, \rho)} \|B x\|_2 \geq c_3 \sqrt{N}) \geq 1 - e^{-c_3 N}.  
$$
Thus, by (\ref{eq:intersection}), $H^\perp \cap \Comp(\d, \rho)  = \emptyset$ with probability at least $ 1- e^{-c_3 N}$.
\end{proof}

Fix the values of $\d$ and $\rho$ for the rest of this section.  We decompose the incompressible vectors into level sets, $S_D$ by the value of the essential least common denominator.  For each level set except those where $D$ is exponentially large, we show that $\inf_{x \in S_D} \|Bx\|_2 >0$.  

Let $\a =\mu \sqrt{N}$, where $\mu >0$ is a small number to be chosen later, which depends only on the subgaussian moment $B$.  By Lemma \ref{lem:lowerboundLCD}, 
$$
\LCD_{\a, c} \geq c_0 \sqrt{N} \qquad \text{ for every } x \in \Incomp(\d, \rho).
$$

\begin{definition}[Level Sets]
Let $D \geq c_0 \sqrt{N}$.  Define $S_D \subseteq S_{\C}^{n-1}$ as
$$
S_D := \{x \in \Incomp: D \leq \LCD_{\a, c}(x) < 2D \}.
$$
\end{definition}

We first derive a lower bound for $\|Bx\|_2$ for a fixed vector $x$.
\begin{lemma} \label{lem:singlevector}
Let $x \in S_D$.  Then for every $t > 0$ we have\
\begin{equation} \label{eq:singlevector}
\P(\|Bx\|_2 < t \sqrt{N}) \leq \left( Ct^2 + \frac{C}{D} + C e^{-c \a^2} \right)^{N-m}
\end{equation}
\end{lemma}
\begin{proof}
We examine the coordinates of $Bx$.
$$
\underline{(Bx)_j} = [x]^T \underline{B_j^T} = \sum_{k=1}^{2N} \xi_k [x]_j. 
$$
Since $x \in S_\C^{N-1}$, we have 
$\sum_{k=1}^{2N} \< [x]_j, y \> ^2 = \|y\|_2^2 \text{ for every } y \in \R^2$.   
We can apply Theorem \ref{thm:smallball} with $m = 2$. 
$$
\P\big( |(Bx)_j | < t \big) \leq Ct^2 + \frac{C}{D} + C e^{-c \a^2}.
$$ 
Since the rows of $B$ are independent, we can use the Tensorization Lemma 2.2 of \cite{RVLittlewoodOfford} to conclude that
$$
\P \big( \sum_{j=1}^{N-m} |(Bx)_j|^2 \leq t^2 (N-m)  \big) \leq \left(C'' t^2 + \frac{C''}{D} + C'' e^{-c \a^2} \right)^{N-m}.
$$
This completes the proof since $\sum_{j=1}^{N-m} |(Bx)_j|^2 = \|B x\|_2^2$.
\end{proof}

We recall the following bound from \cite{luh2018complex} on the size of an $\e$-net of a level set.

\begin{lemma}[Lemma 5.14, \cite{luh2018complex}] \label{lem:smallnet}
There exists a $(2 \a/D)$-net of $S_D$ of cardinality at most $D^2 (C_0 D/\sqrt{N})^{2N}$.
\end{lemma}

Using this bound on the net size and our anti-concentration for a single vector, we can generate a lower bound for an entire level set.
\begin{lemma}[Lower bound for a level set] \label{lem:lowerboundlevelset}
There exist $c_1, c_2, \mu \in (0,1)$ such that the following holds.  Let $\a = \mu \sqrt{N} \geq 1$ and $D \leq c_1 \sqrt{N} e^{c_1 N/m}$.  Then
$$
\P \big(\inf_{x \in S_D} \|B x\|_2 < c_2 N/D \big) \leq 2 e^{-N}.
$$
\end{lemma}

\begin{proof}
By Proposition \ref{prop:opnorm}, there exists $K \geq 1$ such that 
$$
\P(\|B\| > K \sqrt{N}) \leq e^{-N}.
$$
To complete the proof, it suffices to find $\nu > 0$ which depends only on $B$ such that the event
$$
\EE := \left \{ \inf_{x \in S_D} \|B x\|_2 < \frac{\nu N}{2D} \text{ and } \|B\| \leq K \sqrt{N} \right \}
$$
has probability at most $e^{-N}$.

We verify that this holds with the following choice of parameters:
$$
\nu = \frac{1}{(3CC_0)^3 e}, \qquad \mu = \frac{\nu}{9K}, \qquad c_1 = c \mu^2 \leq \nu.
$$

Choosing $\tilde{c}$ in the statement of Theorem \ref{thm:structure} to be sufficiently small, we can assume that $N > \nu^{-2}$.  We apply Lemma \ref{lem:singlevector} with $t = \nu \sqrt{N}/D$.  By our choice of parameters, the $Ct$ term dominates in the right hand side of (\ref{eq:singlevector}).  Therefore, for $x_0 \in S_D$, 
$$
\P\left(\|B x_0\|_2 < \frac{\nu N}{D} \right)  \leq \left(\frac{3 C \nu \sqrt{N}}{D} \right)^{2 (N-m)}.
$$
By Lemma \ref{lem:smallnet}, there exists a $(2\a/D)$-net, $\NN$, of size at most $D^2 (C_0 D/\sqrt{N})^{2N}$
$$
p := \P\Big ( \inf_{x_0 \in \NN} \|B x_0\|_2 < \frac{\nu N}{D} \Big) \leq D^2 (C_0 D/\sqrt{N})^{2N}  \left(\frac{3 C \nu \sqrt{N}}{D} \right)^{2 (N-m)}.
$$
Denote $C_1 := 3 C C_0$.  
$$ 
p \leq C_1^{2N} D^2  \left( \frac{D}{\sqrt{N}} \right)^{2m} \nu^{2(N-m)} \leq C_1^{2N} D^2 (\nu e^{\nu N/m})^{2m}\nu^{N-m} \leq C_1^{3N} \nu^N = e^{-N}.
$$
We assume that $\EE$ occurs.  Fix a $x \in S_D$ for which $\|Bx\|_2 < \frac{\nu N}{2D}$.  There exists an element $x_0 \in \NN$ such that $\|x  - x_0\|_2 \leq \frac{2 \mu \sqrt{N}}{D}.$  Therefore, by the trianlge inequality,
$$
\|B x_0\|_2 \leq \|B x\|_2 + \|B\| \|x - x_0\|_2 \leq \frac{\nu N}{2D} + K \sqrt{N} \frac{2 \mu \sqrt{N}}{D} < \frac{\nu N}{D}.
$$
\end{proof}

\begin{proof}[Proof of Theorem \ref{thm:structure}]
Consider $x \in S^{N-1}_\C$ such that
$$
\LCD_{\a, c}(x) < c_1 \sqrt{N} e^{c_1N/m},
$$
where $c_1$ is the contant from Lemma \ref{lem:lowerboundlevelset}.  Either $x$ is compressible or $x \in S_D$ for some $D \in \DD$, where
$$
\DD := \{D: c_0 \sqrt{N} \leq D < c_1 \sqrt{N} e^{c_1 N/m}, D = 2^k, k \in \N \}.  
$$
We can now decompose the desired probability as
\begin{align*}
p &:= \P\big(\LCD_{\a, c}(H^\perp) < c_1 \sqrt{N} e^{c_1 N/m} \big) \\
&\leq \P(H^\perp \cap \Comp \neq \emptyset) + \sum_{D \in \DD} \P(H^\perp \cap S_D \neq \emptyset).
\end{align*}

By Lemma \ref{lem:randomisincompressible}, the first term on the right is bounded by $e^{-c N}$.  By Lemma \ref{lem:lowerboundlevelset} each term in the summation on the right can be bounded by $2e^{-N}$.  Since $|\DD| \leq C' N$, we have
$$
p \leq e^{-c N} + C' N e^{-N} \leq e^{-c'N}.
$$ 
\end{proof}

\subsection{Invertibility via uniform distance bounds}
The remainder of the proof is identical to \cite{RVRectangle} and is included with the obvious modifications for the reader's convenience.  
We first make several reductions.  Without loss of generality, we may assume that our random variables have an absolutely continuous distribution.  Indeed, we can add to each entry an independent complex gaussian random variable with small variance $\sigma$ and later let $\sigma$ tend to zero.  

Let $N = n-1+d$ for some $d \geq 1$.  We can assume that 
\begin{equation} \label{eq:sizeofd}
1 \leq d \leq c_0 n,
\end{equation}
as when $d$ is above a constant proportion of $n$, our matrix is sufficiently rectangular for a simple epsilon argument (cf. Introduction of \cite{RVRectangle}).  Note that
$$
\sqrt{N} - \sqrt{n-1} \leq \frac{d}{ \sqrt{n}}.
$$
Therefore, 
\begin{multline}                        \label{two terms}
  \P \Big( s_n(A - \l) \le \e \big (\sqrt{N} - \sqrt{n-1} \, \big) \Big)
  \le \P \big( s_n(A-\l) \le \e \frac{d}{\sqrt{n}} \big) \\
  \le \P \big( \inf_{x \in \Comp(\d,\rho)} \|(A -\l) x\|_2 \le \e \frac{d}{\sqrt{n}} \big)
    + \P \big( \inf_{x \in \Incomp(\d,\rho)} \|(A- \l) x\|_2 \le \e \frac{d}{\sqrt{n}} \big).
\end{multline}
We can conclude from Lemma \ref{lem:compressible} that 
\begin{equation}\label{eq: compressible}
  \P \Big( \inf_{x \in \Comp(\d,\rho)} \|(A- \l) x\|_2 \le\e \frac{d}{\sqrt{n}} \Big)
  \le e^{-c_3 N}.
\end{equation}
Therefore, in this section, we focus on a lower bound for incompressible vectors.

Let $X_1, \dots, X_{n} \in \C^{N}$ be the columns of the matrix $A$.  Given a subset $J \subseteq [n]$ of cardinality $d$, we consider the subspace 
$$
H_J := \Span(X_k )_{k \in J} \subseteq \C^{N}.
$$
For levels $K_1, K_2 > 0$ that only depend on $\d, \rho$, we define the set of totally spread vectors
\begin{equation}                        \label{SJ}
  \Spread_J:= \Big\{ y \in S(\C^J) : \;
    \frac{K_1}{\sqrt{d}} \le |y_k| \le \frac{K_2}{\sqrt{d}}
    \quad \text{for all $k \in J$} \Big\}.
\end{equation}

In the following lemma, we let $J$ be a random subset uniformly distributed
over all subsets of $[n]$ of cardinality $d$. To avoid
confusion, we often denote the probability and expectation over
the random set $J$ by $\P_J$ and $\E_J$, and with respect to the
random matrix $A$ by $\P_A$ and $\E_A$.

\begin{lemma}[Total spread]                     \label{lem:totalspread}
  For every $\d,\rho \in (0,1)$,
  there exist $K_1, K_2, c_0 > 0$ which depend only on $\d,\rho$,
  and such that the following holds.
  For every $x \in \Incomp(\d,\rho)$, the event
  $$
  \EE(x) : = \Big\{ \frac{P_J x}{\|P_J x\|_2} \in \Spread_J
  \quad \text{and} \quad
  \frac{\rho \sqrt{d}}{\sqrt{2n}} \le \norm{P_J x}_2
  \le \frac{\sqrt{d}}{\sqrt{\d n}}
  \Big\}
  $$
  satisfies $\P_J(\EE(x)) > c_0^d$.
\end{lemma}

\begin{proof}
Let $\sigma \subset [n]$ be the subset from Lemma \ref{lem:spreadset}.  By choosing $c_0$ sufficiently small in (\ref{eq:sizeofd}), we may assume that $d \leq |\sigma|/2$.  By Stirling's approximation,
$$
\P_J ( J \subset \sigma) = \binom{|\sigma|}{d} \Big / \binom{n}{d} > \left( \frac{\nu_1}{e} \right)^d = c_0^d.
$$ 
Lemma \ref{lem:spreadset} also provides the two-sided bound on $\|P_J x\|_2$.  Thus, we can set $K_1 = \nu_2/ \nu_3$ and $K_2 = 1/K_1$.
\end{proof}

We recall the following lemma from \cite{RVRectangle}.  Although the lemma in \cite{RVRectangle} is stated for real vector spaces, the same proof carries over for complex vector spaces.  
\begin{lemma}[Lemma 6.2, \cite{RVRectangle}] \label{lem:inverttodistance}
There exist $C_1, c_1 >0$ which depend only on $\d, \rho$, and such that the following holds.  Let $J$ be any $d$-element subset of $[n]$.  Then for every $\e > 0$
\begin{equation} \label{eq:distancetospan}
\P\Big( \inf_{x \in \Incomp(\d, \rho)} \|(A- \lambda) x\|_2 < c_1 \e \sqrt{\frac{d}{n}} \Big) \leq C_1^d \cdot \P \big( \inf_{z \in \Spread_J} dist((A - \lambda)z, H_{J^c}) < \e\big).
\end{equation}
\end{lemma}

\subsection{Uniform distance bound}
In this section we bound the probability in the right hand side of (\ref{eq:distancetospan}) following \cite{RVRectangle}.
\begin{theorem}[Uniform distance bound] \label{thm:uniformdistance}
For every $t>0$,
$$
\P \Big( \inf_{z \in \Spread_J} \dist ((A - \l)z, H_{J^c}) < t \sqrt{d} \Big) \leq (Ct)^{2d-1} + e^{-cN}.
$$
\end{theorem}

Since $H_{J^c}$ is the span of $n-d$ independent random vectors and the distribution of the vectors is uniformly continuous, we can assume that
$$
\dim (H_{J^c}) = n-d.
$$
Without loss of generality, in the proof of Theorem \ref{thm:uniformdistance}, we can assume that 
\begin{equation} \label{eq:t}
t \geq t_0 = e^{-\bar{c}N/d}.
\end{equation}
Let us now represent the distance problem in matrix notation. Let $P$ be the orthogonal projection in $\C^N$ onto $(H_{J^c})^\perp$, and let 
\begin{equation} \label{eq:W}
W := PA |_{\C^J}.
\end{equation}
Then for every $v \in \C^N$, the following identity holds: 
\begin{equation} \label{eq:distanceW}
\dist((A-\lambda)z, H_{J^c} + v) = \|Wz - w\|_2, \qquad \text{ where } w = P(v+ \lambda z).
\end{equation}

We omit the standard proof to the following proposition.
\begin{proposition}[Proposition 7.3, \cite{RVLittlewoodOfford}] \label{prop:normW}
$$
\P(\|W\| > t \sqrt{d}) \leq e^{-c_0 t^2 d} \qquad \text{for } t \geq C_0.
$$
\end{proposition}

Having controlled the operator norm of $W$, we can run through the standard approximation argument to uniformly control the distance.
\begin{proposition} \label{prop:uniformdistWsmall}
For every $t$ that satisfies (\ref{eq:t}) we have
\begin{equation} \label{eq:smallWzandW}
\P \Big(\inf_{z \in \Spread_J} \|Wz - w\|_2 < t \sqrt{d} \text{ and } \|W\| \leq K_0 \sqrt{d} \Big)\leq (C_2 t)^{2d-1}.
\end{equation}
\end{proposition}
\begin{proof}
Let $\e = t/K_0$.  By Proposition \ref{prop:nets}, there exists an $\e$-net $\NN$ of $\Spread_J \subseteq S(\C^J)$ of cardinality
$$
|\NN| \leq 4d \left( a + \frac{2}{\e} \right)^{2d -1} \leq 4d \left( \frac{3K_0}{t} \right)^{2d-1}.
$$
Consider the event
$$
\EE := \left \{ \inf_{z \in \NN} \|W z -w\|_2 < 2t \sqrt{d} \right\}.
$$
Taking a union bound, we obtain
$$
\P(\EE) \leq |\NN| \max_{z \in \NN} \P(\|Wz -w\|_2 \leq 2t \sqrt{d}) \leq 4d\left( \frac{3K_0}{t} \right)^{2d-1} (2C_1t)^{4d-2} \leq (C_2 t)^{2d-1}.
$$
Now, suppose the event in (\ref{eq:smallWzandW}) holds, i.e. there exists $z' \in \Spread_J$ such that 
$$
\|Wz' -w\|_2 < t \sqrt{d} \text{ and } \|W\| \leq K_0 \sqrt{d}.
$$  
Choose $z \in \NN$ such that $\|z - z'\|_2 \leq \e$.  Then by the triangle inequality
$$
\|W z - w\|_2 \leq \|W z' - w\|_2 + \|W\| \|z- z'\|_2 < t \sqrt{d} + K_0 \sqrt{d} \e \leq 2t \sqrt{d}.
$$
\end{proof}

We now invoke a proposition from \cite{RVRectangle} which allows us to decouple the behavior of $\|W\|$ and $\|Wz\|_2$.  The proof is a simple translation of the real version.
  
\begin{proposition}[Decoupling, Proposition 7.5, \cite{RVRectangle}]                 \label{prop:decoupling}
  Let $W$ be an $N \times d$ matrix whose columns are independent
  random vectors. Let $\b > 0$ and
  let $z  \in S^{d-1}$ be a vector satisfying $|z_k| \ge
  \frac{\b}{\sqrt{d}}$ for all $k \in \{1 \etc d\}$.
  Then for every $0 < a < b$, we have
  $$
  \P \big( \|Wz\|_2 < a, \; \|W\| > b \big)
    \le 2 \sup_{x \in S^{d-1}_\C, w \in \C^N}
      \P \Big( \|Wx-w\|_2 < \frac{\sqrt{2}}{\b} a \Big) \;
      \P \Big( \|W\| > \frac{b}{\sqrt{2}} \Big).
  $$
\end{proposition} 

We apply this proposition to prove the following lemma.
\begin{lemma}                       \label{lem:uniform distance W}
  Let $W$ be a random matrix as in \eqref{eq:W}, where $P$ is the orthogonal projection
  of $\C^N$ onto the random subspace $(H_{J^c})^\perp$, defined as
  in Theorem \ref{thm:uniformdistance}.
  Then for every $s \ge 1$ and every $t$ that satisfies \eqref{eq:t}, we have
  \begin{align}                  \label{eq: uniform distance W}
    &\P \Big( \inf_{z \in \Spread_J} \|Wz\|_2 < t \sqrt{d}
    \text{ and } s K_0 \sqrt{d} < \|W\| \le 2 s K_0 \sqrt{d} \Big)
    \\
    &\le (C_3 t e^{-c_3 s^2})^{2d-1} +e^{-cN}. \notag
  \end{align}
\end{lemma} 

\begin{proof}
Let $\e = t/2s K_0$.  By Proposition \ref{prop:nets}, there exists an $\e$-net $\NN$ of $\Spread_J \subset S(\R^J)$ of cardinality 
$$
|\NN| \leq 2d \Big(1 + \frac{2}{\e} \Big) ^{2d-1} \leq 2d \left( \frac{6s K_0}{t} \right)^{2d-1}.
$$
Consider the event
$$
\EE := \Big \{ \inf_{z \in \NN} \|Wz\|_2 < 2t \sqrt{d} \text{ and } \|W\| > s K_0 \sqrt{d} \Big \}.
$$
We condition on a realization of the subspace $H_{J^c}$ which allows us to consider the columns of $W$ as independent.  By the definition of $\Spread_J$, we can apply the decoupling proposition \ref{prop:decoupling} with $\beta = K_1$.  Applying a union bound, we have that
\begin{align*}
   \P(\EE \mid H_{J^c})
   &\le |\NN| \cdot \max_{z \in \NN} \P \big( \|Wz\|_2 \le 2t\sqrt{d}
     \text{ and } \|W\| > s K_0 \sqrt{d} \mid H_{J^c} \big) \\
  &\le |\NN| \cdot 2 \max_{z \in S(\R^J), \; w \in \R^N}
    \P \Big( \|Wz- w\|_2 < \frac{\sqrt{2}}{K_1} \cdot 2t\sqrt{d} \mid H_{J^c}
    \Big) \\
    & \quad \cdot \P \Big( \|W\| > \frac{s K_0 \sqrt{d}}{\sqrt{2}} \mid H_{J^c}
    \Big).
\end{align*}
Assuming that $\LCD_{\a, c}(H_{J^c}^\perp) \ge c \sqrt{N}
e^{cN/m}$, where $\a$ and $c$ are as in Theorem \ref{thm:structure},
then by Proposition \ref{prop:normW}  and representation \eqref{eq:distanceW}, we can conclude as in the proof of Theorem \ref{thm:distance} that
$$
\P(\EE  \mid H_{J^c}) \le 8 d \Big( \frac{6 s K_0}{t} \Big)^{2d-1}
  \cdot (C' t)^{4d-2} \cdot e^{-c's^2 d}
$$
 for any $t$ satisfying \eqref{eq:t}.
Since $s \ge 1$ and $d \ge 1$, we can use the following uperbound
$$
\P(\EE  \mid H_{J^c}) \le (C_3 t e^{-c_3 s^2/2})^{2d-1}.
$$
Additionally, by Theorem \ref{thm:structure},
 \begin{align*}
   \P(\EE)
   &\le \P(\EE  \mid \LCD_{\a, c}(H_{J^c}^\perp) \ge c \sqrt{N}
            e^{cN/m} )
   + \P( \LCD_{\a, c}(H_{J^c}^\perp) < c \sqrt{N}  e^{cN/m} ) \\
   &\le (C_3 t e^{-c_3 s^2})^{2d-1} + e^{-cN}.
 \end{align*}
Now, suppose the event in \eqref{eq: uniform distance W} holds.  
There exists $z' \in \Spread_J$ such that
$$
\|Wz'\|_2 < t \sqrt{d} \text{ and } s K_0 \sqrt{d} < \|W\| \le 2 s K_0 \sqrt{d}.
$$
Choose $z \in \NN$ such that $\|z-z'\|_2 \le \e$.
Then by the triangle inequality
$$
\|Wz\|_2 \le \|Wz'\|_2 + \|W\| \cdot \|z-z'\|_2
< t\sqrt{d} + 2 s K_0 \sqrt{d} \cdot \e
\le 2 t \sqrt{d}.
$$
Thus, $\EE$ holds. The conclusion follows from the bound on the probability of $\EE$.
\end{proof}

\begin{proof}[Proof of Theorem \ref{thm:uniformdistance}]
Recall that we can safely assume \eqref{eq:t} holds.
 Let $k_1$ be the smallest natural number such that
 \begin{equation}  \label{k_1}
   2^{k_1} \cdot K_0 \sqrt{d} > C_0 \sqrt{N},
 \end{equation}
 where $C_0$ and $K_0$ are constants from Proposition~\ref{prop:opnorm} and
Lemma~\ref{lem:uniform distance W} respectively.
Summing the probability bounds from Proposition~\ref{prop:uniformdistWsmall} 
and Lemma~\ref{lem:uniform distance W} for $s = 2^k$, $k=1,
\ldots, k_1$, we find that
\begin{align*}
  &\P \Big( \inf_{z \in \Spread_J} \|Wz\|_2 < t \sqrt{d} \Big) \\
  &\le (C_2 t)^{2d-1}
   + \sum_{s = 2^k, \; k=1, \ldots, k_1}
       \Big ( (C_3 t e^{-c_3 s^2})^{2d-1}+ e^{-cN} \Big )
   + \P(\norm{W}>C_0 \sqrt{N})  \\
  &\le (C_4 t)^{2d-1} + k_1 e^{-c'N} + \P(\norm{W}>C_0 \sqrt{N}).
\end{align*}
 By \eqref{k_1} and Proposition~\ref{prop:opnorm}, the last expression is upperbounded by $(C t)^{2d-1} +  e^{-c''N}$.

\end{proof}

\subsection{Proofs of Theorems \ref{thm:rectangular} and \ref{thm:nearlysquare}}
\begin{proof}
By Lemma \ref{lem:inverttodistance} and Theorem \ref{thm:uniformdistance}, we can conclude that
$$\P \big( \inf_{x \in \Incomp(\d,\rho)} \|(A- \l) x\|_2 \le \e \frac{d}{\sqrt{n}} \big) \leq (C \eps)^{2d-1} + e^{-cN}.$$  
By (\ref{two terms}) this concludes the proof.
\end{proof}

A more direct approach suffices for the proof of Theorem \ref{thm:nearlysquare}.  The proof is essentially identical to the square case (c.f. \cite{RVLittlewoodOfford, luh2018complex}).

In this setting, we can use a more straightforward reduction to the distance problem.  
\begin{lemma} [Lemma 3.4, \cite{RVLittlewoodOfford}] \label{lem:invertbydistance}
For $\l \in \C^n$ and $|\l| \leq M \sqrt{N}$,
$$
\P(\inf_{x \in \Incomp(\d, \rho)} \|(A- \l) x\|_2 < \eps \rho n^{-1/2}) \leq \frac{1}{\delta n} \sum_{k=1}^n \P(\dist(X_k, H_k) < \eps)
$$
where $X_k$ denotes the $k$-th column of $A - \lambda$ and $H_k$ is the span of all the columns excluding the $k$-th.
\end{lemma}
\begin{remark}
The proof in \cite{RVLittlewoodOfford} applies equally well in the rectangular setting.
\end{remark}

\begin{proof}[Proof of Theorem \ref{thm:nearlysquare}]
By (\ref{two terms}) and Lemma \ref{lem:invertbydistance}, our task reduces to bounding 
$$
\P(\dist(X_k, H_k) < \eps).
$$ 
By Theorem \ref{thm:distance}, 
$$
\P(\dist(X_k, H_k) < \eps ) \leq (C \sqrt{T} \eps)^{2(N - n+1)} + e^{-cN}.
$$

\end{proof}

\section{Proof of main results} \label{sec:proof}

This section is dedicated to the proof of our main results in Section \ref{sec:main}.  We record the following standard bound for the spectral norm of a random matrix with independent subgaussian entries.  

\begin{lemma} \label{lemma:norm}
Let $A$ be an $N \times n$ genuinely complex random matrix.  There exists constants $M \geq 1$ and $C,c >0$ such that
\[ \Prob( \|A \| \geq M \sqrt{\max\{N, n\}} ) \leq C \exp(-c \max\{N,n\}). \]
Here $M, C, c$ depend only on the uniform subgaussian moment bound $B$.  
\end{lemma}
\begin{proof}
The result essentially follows immediately from \cite[Exercise 2.33]{Tbook}, which applies only to square matrices.  One can easily obtain the bound for rectangular matrices by padding the matrix with zeros to create a square matrix.  Alternatively, one can apply the same net argument as in Proposition \ref{prop:opnorm}.  
\end{proof}

We begin with the proofs of Theorems \ref{thm:eigenvector} and \ref{thm:eigenvector-small}.  

\begin{proof}[Proof of Theorem \ref{thm:eigenvector}]
Without loss of generality, assume $1 \geq t \geq e^{-\log^2 n}$ (as the bound is trivial when $t \geq 1$).  
Let $M \geq 1$ be the constant from Lemma \ref{lemma:norm}.  Let $\eps, \delta$ be positive values to be chosen later, and take
\[ q := \Prob( \loc(A,m,\delta) \text{ and } \|A \| \leq M \sqrt{n} ). \]
Proposition \ref{prop:deloc} implies that
\[ q \leq \frac{9}{\delta^2} \left( \frac{ n e}{m} \right)^m p_o, \]
where $p_0$ satisfies \eqref{eq:deloc:p_0}.  
Choose $\delta$ in terms of $\eps$ via the following identity:
\begin{equation} \label{eq:def:delta}
	6\delta M \sqrt{n} = \eps( \sqrt{n} - \sqrt{n - m -1}). 
\end{equation}
In other words, once we specify $\eps$, $\delta$ will also be determined.  Using Theorem \ref{thm:rectangular}, we find
\[ p_0 \leq (C \eps)^{2m + 1} + e^{-cn}, \]
and hence
\[ q \ll \frac{\eps}{\delta^2} \left( \frac{n e}{m} C^2 \eps^2 \right)^m + \frac{1}{\delta^2} \left( \frac{ne}{m} \right)^m e^{-cn}. \]

Returning to \eqref{eq:def:delta}, we see
\[ \delta \geq \frac{ \eps m}{12 M n}. \]
This implies that
\[ q \ll \left( \left( \frac{n}{m} \right)^{1 + 2/m} e C^2 \eps^{2 - 1/m} \right)^m + \frac{1}{\eps^2} \left( \frac{n}{m} \right)^2 \left( \frac{ n e}{m} \right)^m e^{-cn}. \]
We now choose $\eps$.  Indeed, take
\begin{equation} \label{eq:def:eps}
	\eps := t^{m/(2m-1)} \left( \frac{m}{n} \right)^{(m+2)/(2m-1)}, 
\end{equation}
and recall that this choice of $\eps$ also determines $\delta$ by \eqref{eq:def:delta}.  In addition, this choice implies that
\[ \left( \frac{n}{m} \right)^{1 + 2/m} \eps^{2-1/m} = t, \]
which means
\begin{equation} \label{eq:bndq1}
	q \ll (C^2 e t)^m + \frac{1}{\eps^2} \left( \frac{n}{m} \right)^2 \left( \frac{ n e}{m} \right)^m e^{-cn}. 
\end{equation}

We now simplify the expression for $\eps$ given in \eqref{eq:def:eps} using the fact that $m \geq \log^2 n$.  Indeed, in this case it follows that
\[ \left( \frac{m}{n} \right)^{(m+2)/(2m-1)} = \Theta \left( \left( \frac{m}{n} \right)^{1/2} \right) \]
and, using the fact that $1 \geq t \geq e^{-\log^2 n}$, 
\[ t^{m/(2m-1)} = \Theta( \sqrt{t} ). \]
We conclude that
\begin{equation} \label{eq:eps1}
	\eps = \Theta\left( \sqrt{t} \left( \frac{m}{n} \right)^{1/2} \right), 
\end{equation}
and hence
\begin{equation} \label{eq:delta1}
	\delta \gg \sqrt{t} \left( \frac{m}{n} \right)^{3/2}. 
\end{equation}

Returning to \eqref{eq:bndq1}, we use \eqref{eq:eps1} and $t \geq e^{-\log^2 n}$ to see that 
\[ q \ll (C^2 e t)^m + e^{-c' n} \]
for $m \leq c' n$, where $c' > 0$ is a sufficiently small constant.  
In conclusion, we have now shown that
\[ \Prob( \loc(A,m,\delta) \text{ and } \|A \| \leq M \sqrt{n} ) \ll (C^2 e t)^m + e^{-c' n}  \]
for some $\delta > 0$ which satisfies \eqref{eq:delta1}.  In view of Lemma \ref{lemma:norm}, the proof is complete.  
\end{proof}

\begin{proof}[Proof of Theorem \ref{thm:eigenvector-small}]
The proof is similar to the proof of Theorem \ref{thm:eigenvector}.  Without loss of generality assume $1 \geq t \geq e^{-c' n}$ (as the bound is trivial when $t \geq 1$) for a sufficiently small constant $c' > 0$ to be chosen later.  
Let $M \geq 1$ be the constant from Lemma \ref{lemma:norm}.  Let $\eps, \delta$ be positive values to be chosen later, and take
\[ q := \Prob( \loc(A,m,\delta) \text{ and } \|A \| \leq M \sqrt{n} ). \]
Proposition \ref{prop:deloc} implies that
\[ q \leq \frac{9}{\delta^2} \left( \frac{ n e}{m} \right)^m p_o, \]
where $p_0$ satisfies \eqref{eq:deloc:p_0}.  
Set $\delta$ in terms of $\eps$ again via \eqref{eq:def:delta}, so that $\delta$ is determined completely once we select $\eps$.  Using Theorem \ref{thm:nearlysquare} and the bound $m \leq \log^2 n$, we find that
\[ p_0 \leq (C \eps \log n )^{2(m+1)} + e^{-cn}. \]
Thus, we have
\[ q \ll (\log n)^2 \frac{\eps^2}{\delta^2} \left( \frac{ n e C^2 \log^2 n }{m} \eps^2 \right)^{m} + \frac{1}{\delta^2} \left( \frac{ n e}{m} \right)^m e^{-cn}. \]
From \eqref{eq:def:delta}, we see that
\begin{equation} \label{eq:deltabnd2}
	\delta \geq \frac{\eps}{12 M } \frac{m}{n}, 
\end{equation}
and so
\[ q \ll \left(  \left( \frac{n}{m} \right)^{1 + 2/m} e C^2 (\log n)^{2 + 2/m} \eps^2 \right)^m + \frac{1}{\delta^2} \left( \frac{ n e}{m} \right)^m e^{-cn}. \]

Define $\eps$ by the following identity:
\[ \left( \frac{n}{m} \right)^{1  + 2/m} (\log n)^{2 + 2/m} \eps^2 = t. \]
This implies that
\[ \eps = \frac{ \sqrt{t}}{ (\log n)^{1 + 1/m}} \left( \frac{m}{n} \right)^{(m+2)/2m} \geq \frac{ \sqrt{t}}{ \log^{2} n} \left( \frac{m}{n} \right)^{(m+2)/2m}. \]
In view of \eqref{eq:deltabnd2} we see that
\begin{equation} \label{eq:deltabnd3}
	\delta \gg \left( \frac{m}{n} \right)^{3/2 + 1/m} \frac{ \sqrt{t}}{ \log^{2} n }. 
\end{equation}
In addition, we obtain
\[ q \ll (C^2 e t )^m + \frac{n^5}{t} (\log n)^4 \left( n e \right)^{\log^2 n} e^{-c n}. \]
Using the assumption that $t \geq e^{-c' n}$ and taking $c'$ sufficiently small, we deduce that
\[ q \ll (C^2 e t )^m + e^{-c'' n} \]
for some constant $c'' > 0$.  

In conclusion, we have now shown that
\[ \Prob( \loc(A,m,\delta) \text{ and } \|A \| \leq M \sqrt{n} ) \ll (C^2 e t)^m + e^{-c'' n}  \]
for some $\delta > 0$ which satisfies \eqref{eq:deltabnd3}.  In view of Lemma \ref{lemma:norm}, the proof is complete.  
\end{proof}

We now turn to the proofs of Theorem \ref{thm:eigenvectorreal} and \ref{thm:eigenvector-smallreal}.  We will need the following least singular value bound for real iid matrices, adopted from \cite{RVRectangle}.  
\begin{theorem} \label{thm:RVlsv}
Let $A$ be an $N \times n$ real random matrix, $N \geq n$, whose elements are independent copies of a mean zero subgaussian random variable with unit variance.  Then for every $\eps > 0$ and $\lambda \in \mathbb{R}$ with $|\lambda| \leq M \sqrt{N}$ for some $M \geq 1$, we have
\[ \Prob \left( s_n(A - \lambda) \leq \eps ( \sqrt{N} - \sqrt{n-1} ) \right) \leq (C\eps)^{N-n+1} + e^{-cN} \]
where $C, c > 0$ depend (polynomially) only on the subgaussian moment of the entries and $M$.  
\end{theorem}

The $\lambda = 0$ case of this theorem appears as \cite[Theorem 1.1]{RVRectangle}.  However, a close inspection of their proof confirms that their argument can be adapted to the shifted case, in the same way that we have explicitly done in the proof of Theorem \ref{thm:rectangular}.

\begin{proof}[Proof of Theorem \ref{thm:eigenvectorreal}]
The proof is similar to the proof of Theorem \ref{thm:eigenvector}.  Without loss of generality, assume $1 \geq t \geq e^{-c' n}$ for some constant $c' > 0$ to be chosen later (as the bound is trivial when $t \geq 1$).  By \cite[Proposition 2.3]{RVRectangle}, there exists $M \geq 1$ such that 
\begin{equation} \label{eq:normreal}
	\Prob( \|A\| \leq M \sqrt{n}) \geq 1 - C_0 e^{-c_0 n}, 
\end{equation}
where $M, C_0, c_0 > 0$ depend only on the subgaussian moment of the entries.  
Let $\eps, \delta$ be positive values to be chosen later, and take
\[ q := \Prob( \loc_{\mathbb{R}}(A,m,\delta) \text{ and } \|A \| \leq M \sqrt{n} ). \]
Proposition \ref{prop:delocreal} implies that
\[ q \leq \frac{3}{\delta} \left( \frac{ n e}{m} \right)^m p_o, \]
where $p_0$ satisfies \eqref{eq:deloc:p_02}.  
Choose $\delta$ in terms of $\eps$ via \eqref{eq:def:delta}, and again note that once we specify $\eps$, $\delta$ will also be determined.  Using Theorem \ref{thm:RVlsv}, we find
\[ p_0 \leq (C \eps)^{m + 1} + e^{-cn}, \]
and hence
\[ q \ll \frac{\eps}{\delta} \left( \frac{n e}{m} C \eps \right)^m + \frac{1}{\delta} \left( \frac{ne}{m} \right)^m e^{-cn}. \]

Returning to \eqref{eq:def:delta}, we see
\[ \delta \geq \frac{ \eps m}{12 M n}. \]
This implies that
\[ q \ll \left( \left( \frac{n}{m} \right)^{1 + 1/m} e C \eps \right)^m + \frac{1}{\eps} \left( \frac{ n e}{m} \right)^{m+1} e^{-cn}. \]
We now choose $\eps$.  Indeed, take
\begin{equation} \label{eq:def:eps_real}
	\eps := t \left( \frac{m}{n} \right)^{ (m+1)/m }, 
\end{equation}
and recall that this choice of $\eps$ also determines $\delta$ by \eqref{eq:def:delta}.  In addition, this choice implies that
\begin{equation} \label{eq:bndq1_real}
	q \ll (C e t)^m + \frac{1}{\eps} \left( \frac{ n e}{m} \right)^{m+1} e^{-cn}. 
\end{equation}

We now simplify the expression for $\eps$ given in \eqref{eq:def:eps_real} using the fact that $m \geq \log^2 n$.  Indeed, in this case it follows that
\[ \left( \frac{m}{n} \right)^{(m+1)/(m)} = \Theta \left( \frac{m}{n} \right).  \]
We conclude that
\begin{equation} \label{eq:eps1_real}
	\eps = \Theta\left( t \frac{m}{n} \right), 
\end{equation}
and hence
\begin{equation} \label{eq:delta1_real}
	\delta \gg t \left( \frac{m}{n} \right)^{2}. 
\end{equation}

Returning to \eqref{eq:bndq1_real}, we use \eqref{eq:eps1_real} and $t \geq e^{-c' n}$ to see that 
\[ q \ll (C e t)^m + e^{-c' n} \]
for $m \leq c' n$, where $c' > 0$ is a sufficiently small constant.  
In conclusion, we have now shown that
\[ \Prob( \loc_{\mathbb{R}}(A,m,\delta) \text{ and } \|A \| \leq M \sqrt{n} ) \ll (C e t)^m + e^{-c' n}  \]
for some $\delta > 0$ which satisfies \eqref{eq:delta1_real}.  In view of \eqref{eq:normreal}, the proof is complete.  
\end{proof}

Theorem \ref{thm:eigenvector-smallreal} follows from similar arguments as those presented in the proofs of Theorems \ref{thm:eigenvector-small} and \ref{thm:eigenvectorreal}; we omit the details. 

We now turn to the proofs of results from Section \ref{sec:main:normal}.  

For the proof of Theorem \ref{thm:nogapsasym}, we first recall a result from \cite{NV}.
\begin{theorem}[Theorem 1.4, \cite{NV}] \label{thm:normalcoordinates}
For a $n-1 \times n$ genuinely complex random matrix $A$, let $x$ be a vector normal to all the rows.  There exists a positive constants $c$ and $c'$ such that for any $d$-tuple $(i_1, \dots, i_d)$ with $d = n^c$ and $\Omega \in \C^d$,
$$
|\P((\sqrt{n} x_{i_1}, \dots, \sqrt{n} x_{i_d}) \in \Omega) - \P(\mathbf{g}_{\C, 1}, \dots, \mathbf{g}_{\C, d}) \in \Omega )| \leq d^{-c'}.
$$
\end{theorem}  
We model our proof of Theorem \ref{thm:nogapsasym} after the proof of Theorem 5.1 in \cite{OVWsurvey}.
\begin{proof}[Proof of Theorem \ref{thm:nogapsasym}] 
Let $Z$ and $Z'$ be standard normal distributions with cumulative distribution function $\Phi(x)$.  Recall that $F(x)$ is the cumulative distribution function of $Z^2/2 + Z'^2/2$.  

\begin{align*}
F(x) &= \frac{1}{2 \pi} \int \int_{z^2 + z'^2 \leq 2x} e^{-z^2/2} e^{-z'^2/2} dz dz'  \\
&= \int_0^{\sqrt{2x}} e^{-r^2/2} r dr \\
&= 1 - e^{-x}
\end{align*}

For convenience, we introduce the function 
$$
G(x) := F(x^2) = 1 - e^{-x^2}.
$$
By direct calculation,
\begin{align*}
- \int_{1 - \delta}^1 H(u) du &= \int_0^\delta F^{-1}(u) du \\
&= \int_0^{\sqrt{F^{-1}(\delta)}} 2 x^3 F'(x) dx \\
&= \int_0^{G^{-1}(\delta)} x^2 G'(x) dx \\
&= - \int_0^\delta \log(1-u) du
\end{align*}
and 
$$
- \int_0^\delta H(u) du = \int_{1- \delta}^1 F^{-1}(u) = - \int_{1-\delta}^1 \log(1-u) du.
$$

Thus, it suffices to show
\begin{equation}\label{eq:maxsubset}
\Big|\max_{S \subset [n]: |S| = \lfloor \delta n \rfloor } \|x_S\|_2^2 +  \int_0^\delta \log(1-u) du \Big | \leq \e
\end{equation}
and
\begin{equation} \label{eq:minsubset}
\Big|\min_{S \subset [n]: |S| = \lfloor \delta n \rfloor } \|x_S\|_2^2 +  \int_{1-\delta}^1 \log(1-u) du \Big | \leq \e
\end{equation}
In fact, we can simply focus on (\ref{eq:minsubset}) as (\ref{eq:maxsubset}) follows from the identity
$$
\max_{S \subset [n]: |S| = \lfloor \delta n \rfloor } \|x_S\|_2^2 + \min_{S \subset [n]: |S| = \lfloor \delta n \rfloor } \|x_S^c\|_2^2 = 1
$$
and $\int_0^1 \log (1-u) du = 1$.

Define 
$$
N(c,k) := \sum_{j=1}^n \mathbf{1}_{\{c(k-1) \leq \sqrt{n} |v(j)| < ck \} }
$$
Let $Z$ be a complex gaussian and define
$$
f(c,k) := n \P( c(k-1) \leq  |Z| < ck).
$$
Note that 
$$
f(c,k) =  n \big(G(c k) - G(c (k-1))\big).
$$
By Theorem \ref{thm:normalcoordinates}, we have that 
$$
\P(c(k-1) \leq \sqrt{n} |x(j)| < ck) = \P(c (k-1) \leq |Z| < ck) (1+ o(1))
$$
uniformly for all $1 \leq j \leq n$.  Thus,
$$
\E N(c,k) = (1+ o(1)) f(c,k).  
$$
Similarly, we can verify that
$$
\text{Var}(N(c,k)) = o(n^2).
$$
By Chebyshev's inequality, we can conclude that
\begin{equation} \label{eq:N(c,k)}
N(c,k) = (1+o(1)) f(c,k)
\end{equation}
with probability $1-o(1)$.

We choose $c>0$ and $k_0 \in \N$ so that 
\begin{equation} \label{eq:cdefinition}
 c \Big [G^{-1}(\delta) \Big]^2 < \frac{\eps}{2}
\end{equation}
and 
$ck_0 = G^{-1}(\delta)$.  This definition ensures that
\begin{equation} \label{eq:delta}
\sum_{k=1}^{k_0} G(c k) - G(c (k-1)) = G(c k_0) - G(0) = \delta.
\end{equation}
Additionally, we have
\begin{align} \label{eq:lowereps/2}
&\left| \sum_{k=1}^{k_0} c^2 (k-1)^2   (G(ck) - G(c(k-1))) - \int_0^{G^{-1}(\delta)} x^2 G'(x) dx \right| \nonumber \\
&\qquad = \left|\sum_{k=1}^{k_0} c^2 (k-1)^2 (G(ck) - G(c(k-1))) - \int_{c(k-1)}^{ck} x^2 G'(x) dx \right| \nonumber \\
&\qquad = \left| \sum_{k=1}^{k_0} \left [  (c^2 (k-1)^2 - c^2 k^2) G(ck) + 2 \int_{c(k-1)}^{ck}  x G(x) dx \right] \right| \\
&\qquad = 2 \left| \sum_{k=1}^{k_0} \int_{c(k-1)}^{ck}  x (G(x) - G(ck)) dx  \right| \nonumber \\
&\qquad \leq  c \left[ G^{-1}(\delta) \right]^2 \nonumber \\
&\qquad < \frac{\e}{2} \nonumber 
\end{align}
by integration by parts and (\ref{eq:cdefinition}).  The first inequality follows from 
the mean value theorem, the identity $G'(x) = 2x e^{-x^2}$ and the bound $|G'(x)| \leq \sqrt{2} e^{-2} \leq 1$.
By an identical argument, we can show that
\begin{equation} \label{eq:uppereps/2}
\left| \sum_{k=1}^{k_0} c^2 k^2 (G(ck) - G(c(k-1))) - \int_0^{G^{-1}(\delta)} x^2 G'(x) dx \right| < \frac{\eps}{2}.
\end{equation}
By (\ref{eq:N(c,k)}), for any $1 \leq k \leq k_0+1$,
\begin{equation} \label{eq:f(c,k)}
N(c,k) = (1+o(1)) f(c,k) = (1+ o(1)) 2n  (G(ck) - G(c(k-1)))
\end{equation}
with probability $1-o(1)$.  (\ref{eq:delta}) implies that
$$
\sum_{k=1}^{k_0} f(c,k) = \delta n.
$$ 
Therefore, by a union bound, with probability $1-o(1)$, 
$$
\sum_{k=1}^{k_0} N(c,k) = (1 + o(1)) \delta n = \lceil \delta n \rceil + o(n).
$$

We have the two-sided bound
$$
\sum_{k=1}^{k_0} c^2 (k-1)^2 N(c,k) \leq n \min_{S \subset [n]: |S| = \sum_{k=1}^{k_0} N(c,k) } \|x_S\|_2^2 \leq \sum_{k=1}^{k_0} c^2 k^2 N(c,k).
$$
With probability $1-o(1)$, there exists a sequence $\tau_n$ with $\tau_n \rightarrow 0$ such that 
\begin{align*}
\sum_{k=1}^{k_0} c^2 &(k-1)^2 N(c,k) - \tau_n c^2 k_0^2 N(c,k_0) \\
&\leq \min_{S \subset [n]: |S| = \lceil \delta n \rceil} \|x_S\|_2^2 \leq \sum_{k=1}^{k_0} c^2 k^2 N(c,k) + \tau_n c^2 (k_0+1)^2 N(c, k_0+1).
\end{align*}

By (\ref{eq:f(c,k)}), we also have that
\begin{align*}
\sum_{k=1}^{k_0} c^2 &(k-1)^2 (G(ck) - G(c(k-1))) (1 + o(1)) \\
&\leq \min_{S \subset [n]: |S| = \lceil \delta n \rceil} \|x_S\|_2^2 \leq \sum_{k=1}{k_0} c^2 k^2 (G(ck) - G(c(k-1))) (1+ o(1))
\end{align*} 
with probability $1- o(1)$.  Finally, combining (\ref{eq:lowereps/2}) and (\ref{eq:uppereps/2}), we can conclude that
$$
\left| \min_{S \subset [n]: |S| = \lceil \delta n \rceil} \|x_S\|_2^2 - \int_0^{G^{-1}(\delta)} x^2 G'(x) dx \right| \leq \eps
$$ 
with probability $1-o(1)$.
\end{proof}

\begin{proof}[Proof of Theorem \ref{thm:nogaps}]
The proof closely mirrors the proof of Theorem \ref{thm:eigenvector}.  Without loss of generality, assume $1 \geq t \geq e^{-\log^2 n}$ (as the bound is trivial when $t \geq 1$).  
Let $M \geq 1$ be the constant from Lemma \ref{lemma:norm}.  Let $\eps, \delta$ be positive values to be chosen later, and take
\[ q := \Prob( \loc_{0}(A,M, m,\delta) \text{ and } \|A \| \leq M \sqrt{n} ). \]
(Here, we have set $\lambda_0 = 0$ in the definition of $\loc_{\lambda_0}(A,M,m,\delta)$.)  
Proposition \ref{prop:delocgeneral} implies that
\[ q \leq \left( \frac{ n e}{m} \right)^m p_o, \]
where $p_0$ satisfies \eqref{eq:p_0def1}.  
Choose $\delta$ in terms of $\eps$ via the following identity:
\begin{equation} \label{eq:def:delta2}
	6\delta M \sqrt{n} = \eps( \sqrt{n-1} - \sqrt{n - m -1}). 
\end{equation}
In other words, once we specify $\eps$, $\delta$ will also be determined.  Using Theorem \ref{thm:rectangular}, we find
\[ p_0 \leq (C\eps)^{2m-1} + C e^{-cn}, \]
and so
\[ q \ll \left( \frac{n e C^2 }{m} \eps^{(2m-1)/m} \right)^m + \left( \frac{ne}{m} \right)^m e^{-cn}. \]

Choose $\eps > 0$ such that
\[ \frac{n}{m} \eps^{(2m-1)/m} = t. \]
This choice implies that
\[ q \ll \left( e C^2 t \right)^m + \left( \frac{ne}{m} \right)^m e^{-cn} \]
and
\[ \eps = \left( t \frac{m}{n} \right)^{m/(2m-1)}. \]
We now use the assumption that $m \geq \log^2 n$ to simplify this expression for $\eps$.  Indeed, in this case it follows that
\[ \left( \frac{m}{n} \right)^{m/(2m-1)} = \Theta \left( \left( \frac{m}{n} \right)^{1/2} \right). \]
Similarly, since $1 \geq t \geq e^{-\log^2 n}$, we have
\[ t^{m/(2m-1)} = \Theta( \sqrt{t} ). \]
Thus, we conclude that
\[ \eps = \Theta \left( \sqrt{t} \left( \frac{m}{n} \right)^{1/2} \right). \]
Combining this with \eqref{eq:def:delta2}, we see that
\begin{equation} \label{eq:normaldelta}
	\delta \gg \frac{m}{n} \eps \gg \sqrt{t} \left( \frac{m}{n} \right)^{3/2}. 
\end{equation}

In addition, there exists a sufficiently small constant $c' > 0$ such that 
\[ q \ll \left( e C^2 t \right)^m + \left( \frac{ne}{m} \right)^m e^{-cn} \ll \left( e C^2 t \right)^m + e^{-c' n} \]
for $m \leq c' n$.  To conclude, we have shown that
\[ \Prob( \loc_{0}(A,M, m,\delta) \text{ and } \|A \| \leq M \sqrt{n} ) \ll \left( e C^2 t \right)^m + e^{-c' n} \]
for some $\delta$ which satisfies \eqref{eq:normaldelta}.  In view of Lemma \ref{lemma:norm}, the proof is complete.  
\end{proof}

\begin{proof}[Proof of Theorem \ref{thm:nogaps-small}]
The proof follows closely the proofs of Theorems \ref{thm:eigenvector-small} and \ref{thm:nogaps}.  Without loss of generality, assume $1 \geq t > 0$ (as the bound is trivial when $t \geq 1$).  
Let $M \geq 1$ be the constant from Lemma \ref{lemma:norm}.  Let $\eps, \delta$ be positive values to be chosen later, and take
\[ q := \Prob( \loc_{0}(A,M, m,\delta) \text{ and } \|A \| \leq M \sqrt{n} ). \]
Proposition \ref{prop:delocgeneral} implies that
\[ q \leq \left( \frac{ n e}{m} \right)^m p_o, \]
where $p_0$ satisfies \eqref{eq:p_0def1}.  
Again take $\delta$ in terms of $\eps$ by \eqref{eq:def:delta2}, so that $\delta$ is completely determined once we specify $\eps$.  Applying Theorem \ref{thm:nearlysquare} and using the fact that $1 \leq m \leq \log^2 n$, we find
\[ p_0 \leq (C \eps \log n)^{2m} + Ce^{-cn}. \]
This gives
\[ q \ll \left( \frac{n}{m} e C^2 (\log n)^2 \eps^2 \right)^m + \left( \frac{ne}{m} \right)^m e^{-cn}. \]
Define 
\[ \eps := \frac{\sqrt{t}}{\log n} \sqrt{ \frac{m}{n} }, \]
so that
\[ \eps^2 \frac{n}{m} \log^2 n = t. \]
From \eqref{eq:def:delta2}, we see that this choice of $\eps$ gives
\begin{equation} \label{eq:nogaps-small:delta}
	\delta \gg \frac{m}{n} \eps \geq \frac{\sqrt{t}}{\log n} \left( \frac{m}{n} \right)^{3/2}. 
\end{equation}
In addition, it follows that 
\[ q \ll \left( e C^2 t \right)^m + \left( \frac{ne}{m} \right)^m e^{-cn}. \]

Since $1 \leq m \leq \log^2 n$, 
\[ \left( \frac{ne}{m} \right)^m e^{-cn} \ll e^{-c' n} \]
for some sufficiently small constant $c' > 0$.  We conclude that
\[ \Prob( \loc_{0}(A,M, m,\delta) \text{ and } \|A \| \leq M \sqrt{n} ) \ll \left( e C^2 t \right)^m + e^{-c'n} \]
for some $\delta$ which satisfies \eqref{eq:nogaps-small:delta}.  In view of Lemma \ref{lemma:norm}, the proof is complete.  
\end{proof}

\appendix

\section{Proof of Propositions \ref{prop:largest-coord} and \ref{prop:uniform-vector}} \label{sec:uniform-vector}

\begin{proof}[Proof of Proposition \ref{prop:largest-coord}]
We prove the result for $\R^n$, but an analogous argument applies in $\C^n$.  We model the uniform distribution on the unit sphere by sampling a gaussian vector $g \sim N(0, I_n)$ and normalizing by $\|g\|_2^{-1}$.  Let $\EE$ denote the event that $\|g\|_2 \geq \frac{\sqrt{n}}{10}$.  By standard concentration bounds, we have that
$$
\P(\EE^C) \leq \exp(-c n).
$$
Similarly, for $1 \leq i \leq n$,
$$
\P(|g_i| \geq t) \leq 2 \exp(-t^2/8).
$$
Therefore, 
\begin{align*}
\P\left(|v_i| \geq C \sqrt{\frac{\log n}{n}}\right) &= \P\left(\frac{|g_i|}{\|g\|_2} \geq C \sqrt{\frac{\log n}{n}}\right) \\
&\leq \P(|g_i| \geq C \sqrt{\log n}) + \P(\EE^C) \\
&\leq 1/n^2
\end{align*}
for large enough $C$.  Therefore, applying the union bound, 
$$
\|v\|_{\infty} \leq C \sqrt{\frac{\log n}{n}}
$$
with probability $1- o(1)$.    
\end{proof}

\begin{proof} [Proof of Proposition \ref{prop:uniform-vector}]
We address the complex case first.  As we are not trying to optimize the constant in the exponent of the logarithm, we can conveniently assume that $C \log n \leq m \leq n/ \log n$ for any constant $C$.  We follow the convention that $C, c$ denote absolute constants that may change from line to line.  Again, we model the uniform distribution on the unit sphere in $\C^n$ by considering a random variable $g \sim N_{\C}(0,1)$ that is normalized by $\|g\|_2^{-1}$.  Let $\mathcal{E}$ denote the event that $\|g\|_2 \leq 10 \sqrt{n}$.  We have that
$$
\P(\mathcal{E}^c) \leq \exp(-c n).  
$$
We let $Y_1< \dots < Y_n$ denote the order statistics of the magnitudes of $|g_1|, \dots, |g_n|$.
Therefore,
\begin{align} \label{eq:lowerboundm}
\P\Big(\|v_I\|_2 \leq \frac{C}{\log^c n} \frac{m}{n} \text{ for all } I \subset [n], |I|=m \Big) &\leq \P\Big(\|v_I\|_2 \leq \frac{C}{\log^c n}  \frac{m}{n} \text{ for all } I \text{ and } \mathcal{E} \Big) + \P(\mathcal{E}^c) \\
&\leq \P\left(\sum_{i=1}^m Y_i^2 \leq  \frac{C}{\log^c n}  \frac{m^2}{n} \right) + \exp(-cn). \nonumber
\end{align}
We use a simple counting and grouping argument to control the latter probability.  
We define the following random variables that count the number of coordinates with magnitude in a fixed range.
$$
\eta_k := \sum_{i=1}^n \mathbbm{1}_{\left \{ \frac{\delta}{n} 2^{k-1} \leq |g_i|^2 \leq \frac{\delta}{n} 2^k \right \}}
$$
for $1 \leq k \leq L$ where $\delta := 1/ \log n$  and $L = \lfloor \log_2(m/2\delta) \rfloor$.  Additionally, we denote the probability of a coordinate falling in this range by 
$$
p_k := \P\left( \frac{\delta}{n} 2^{k-1} \leq |g_i|^2 \leq \frac{\delta}{n} 2^k  \right).
$$
As $\eta_k$ is the sum of independent random variables, we have that
$$
\E \eta_k = n p_k
$$
and 
$$
\text{Var}(\eta_k) = n p_k (1 - p_k) \leq n p_k.
$$
By Chebyshev's inequality, for $t > 0$,
\begin{equation} \label{eq:chebyshev}
\P(|\eta_k - \E \eta_k| > t) \leq \frac{n p_k}{t^2}.
\end{equation}
As $|g_1|^2$ is a chi-squared distribution with two degrees of freedom and $\frac{\delta}{n} 2^k \leq 1$ for all $k$, by the bounded density of the chi-squared distribution, we deduce that
\begin{equation} \label{eq:boundeddensity}
p_k = \Theta\left(\frac{\delta}{n} 2^{k-1}\right).
\end{equation}
Let $\mathcal{E}'$ denote the event that
\begin{equation}\label{eq:controleta}
\eta_k \geq c\delta 2^{k-1} - 2^{(2/3)k}  
\end{equation}
for all $1 \leq k \leq L$.
Combining (\ref{eq:chebyshev}) and (\ref{eq:boundeddensity}), we can conclude that $P(\mathcal{E}') \geq 1 - O(\delta)$.  In particular, observe that for $k \geq \log \log n$, say, (\ref{eq:controleta}) implies that $\eta_k \geq \delta 2^{k-2}$ for large enough $n$.

Recall that the cumulative distribution function of a chi-squared distribution with two degrees of freedom is $F(x) = 1 - e^{-x/2}$ for $x \geq 0$.  Therefore, from our choice of $L = \lfloor \log_2(m/2\delta) \rfloor$ we find that
$$
\frac{\delta 2^{L}}{n} \leq \frac{m}{2n} \leq -2 \log(1-m/2n)
$$
Using the cumulative distribution function, we find that 
$$
\P\left(|g_1|^2 \leq \frac{\delta 2^{L}}{n}\right) \leq \P(|g_1|^2 \leq -2 \log(1-m/2n) = m/2n.
$$ 
Thus, by Chernoff's bound, 
$$
P(Y_m \geq \delta 2^{L}) \leq 2 e^{-c m} = O(\delta),
$$
where the last equality follows from the assumptions that $m \geq C \log n$ and $\delta > 1/n$.  
Finally, we have that with probability at least $1 - O(\delta)$, 
$$
\sum_{i=1}^m Y_i^2  \geq \sum_{k=\log \log n}^L \delta 2^{k-1} \frac{\delta}{n} 2^{k-1} = \Omega\left(\frac{\delta^2 2^{2L}}{n}\right) = \Omega\left(\frac{m^2}{n \log^2 n}\right).
$$
We have shown that 
$$
\P\left(\sum_{i=1}^m Y_i^2 \leq  \frac{C}{\log^c n}  \frac{m^2}{n} \right) = O(\delta).
$$
From \ref{eq:lowerboundm}, we can infer that
$$
\P\Big(\|v_I\|_2 \leq \frac{C}{\log^c n} \frac{m}{n} \text{ for all } I \subset [n], |I|=m \Big) = O(\delta) = o(1),
$$
which concludes the proof of the complex case. \newline 


The real case follows the same outline. The choice of parameters is slightly different as the density of the chi squared distribution with one degree of freedom no longer has bounded density but grows as $x^{-1/2}$ near zero.  We use the same notation as in the complex case.   

Again, we can assume that $C \log n \leq m \leq n/ \log n$ for any large constant $C$.  We model the uniform distribution on the sphere in $\R^n$ by considering a random variable $g \sim N_{\mathbb{R}}(0,1)$ that is normalized by $\|g\|_2^{-1}$.  Let $\mathcal{E}$ denote the event that $\|g\|_2 \leq 10 \sqrt{n}$.  We have that
$$
\P(\mathcal{E}^c) \leq \exp(-c n).  
$$
We let $Y_1< \dots < Y_n$ denote the order statistics of the magnitudes of $|g_1|, \dots, |g_n|$.
Therefore,
\begin{align} \label{eq:reallowerboundm}
\P\Big(\|v_I\|_2 \leq \frac{C}{\log^c n} \left(\frac{m}{n} \right)^{3/2} &\text{ for all } I \subset [n], |I|=m \Big) \\
 &\leq \P\Big(\|v_I\|_2 \leq \frac{C}{\log^c n}  \left(\frac{m}{n} \right)^{3/2} \text{ for all } I \text{ and } \mathcal{E} \Big) + \P(\mathcal{E}^c) \nonumber\\ 
&\leq \P\left(\sum_{i=1}^m Y_i^2 \leq  \frac{C}{\log^c n}  \frac{m^3}{n^2} \right) + \exp(-cn). \nonumber
\end{align}  
To control the latter probability, we define the following random variables that count the number of coordinates with magnitude in a fixed range.
$$
\eta_k := \sum_{i=1}^n \mathbbm{1}_{\left \{ \frac{\delta^2}{n^2} 2^{k-1} \leq |g_i|^2 \leq \frac{\delta^2}{n^2} 2^k \right \}}
$$
for $1 \leq k \leq L$ where $\delta := 1/ \log n$  and $L = \lfloor 2 \log_2(m/C^*\delta) \rfloor$ where $C^*$ is a constant to be fixed later.  We denote the probability of a coordinate falling in this range by 
$$
p_k := \P\left( \frac{\delta^2}{n^2} 2^{k-1} \leq |g_i|^2 \leq \frac{\delta^2}{n^2} 2^k  \right).
$$
As $\eta_k$ is the sum of independent random variables, we have that
$$
\E \eta_k = n p_k
$$
and 
$$
\text{Var}(\eta_k) = n p_k (1 - p_k) \leq n p_k.
$$

As a chi-squared distrbuted random variable with one degree of freedom has probability density function $\Theta(x^{-1/2})$ near zero and $\frac{\delta^2}{n^2} 2^k \leq 1$ for all $k$,  we deduce that
\begin{equation} \label{eq:realboundeddensity}
p_k = \Theta\left(\frac{\delta}{n} 2^{k/2}\right).
\end{equation}
Let $\mathcal{E}'$ denote the event that
\begin{equation}\label{eq:realcontroleta}
c \delta 2^{k/2} - 2^{ k/3} \leq \eta_k \leq C \delta 2^{k/2} + 2^{ k/3} 
\end{equation}
for all $1 \leq k \leq L$ and $C, c$ are the implied constants in  (\ref{eq:realboundeddensity}).
By Chebyshev's inequality and (\ref{eq:realboundeddensity}), we can conclude that the probability of (\ref{eq:realcontroleta}) is larger than $1 - O(\delta)$.  In particular, observe that for $k \geq \log \log n$, say, (\ref{eq:controleta}) implies that $\eta_k \geq c \delta 2^{k/2}$ for large enough $n$.

By our choice of $L$, the probability that $\sum_{i=0}^L \eta_k \geq m$ is at most $O(\delta)$ for large enough $C^*$.  By a simple calculation, we also have that $P(Y_1 < \delta^2/n^2) = O(\delta)$.
Therefore, we can have shown that 
\begin{align*}
\P\left(\sum_{i=1}^m Y_i^2 \geq  \sum_{k = \log \log n}^{L} \delta 2^{k/2} \frac{\delta^2 2^k}{n^2}  \right) 
&\geq \P\left(\sum_{i=1}^m Y_i^2 \geq   \frac{C \delta^3 2^{3L/2}}{n^2}  \right) \\
&\geq \P\left(\sum_{i=1}^m Y_i^2 \geq   \frac{C m^{3}}{n^2 \log^c n}  \right) \\
&=   1 - O(\delta).
\end{align*}
From \ref{eq:reallowerboundm}, we can conclude that
$$
\P\Big(\|v_I\|_2 \leq \frac{C}{\log^c n} \left(\frac{m}{n}\right)^{3/2} \text{ for all } I \subset [n], |I|=m \Big) = O(\delta) = o(1),
$$
which finishes the proof of the real case.

\end{proof}

\bibliographystyle{abbrv}
\bibliography{delocalization}

\end{document}